\documentclass[11pt]{article}
\usepackage[T1]{fontenc}
\usepackage{lmodern}
\usepackage{amsmath,amsfonts,amssymb,amsthm}

\usepackage{graphicx}
\usepackage{multirow}
\usepackage{longtable,booktabs}
\usepackage{subcaption}
\usepackage{authblk}
\usepackage[a4paper,top=2.45cm,bottom=2.45cm,left=2.7cm,right=2.7cm,marginparwidth=2cm]{geometry}
\usepackage{float}
\usepackage{makecell}

\usepackage{natbib}
 \bibpunct[,]{(}{)}{;}{a}{,}{,}

\usepackage[dvipsnames,table,svgnames]{xcolor}
\usepackage{hyperref}
\hypersetup{
	colorlinks,
	linkcolor = {NavyBlue},
	citecolor = {NavyBlue},
	urlcolor = {NavyBlue}
}

\definecolor{dark-red}{rgb}{0.8,0.0,0.0}

\usepackage[flushleft]{threeparttable}
\usepackage{multirow}
\usepackage{array}
\usepackage[export]{adjustbox}

\usepackage{tikz}

\newtheorem{theorem}{Theorem}

\newtheorem{lemma}[theorem]{Lemma}
\newtheorem{corollary}[theorem]{Corollary}

\theoremstyle{definition}

\theoremstyle{remark}

\providecommand{\keywords}[1]
{\small\textbf{Keywords:} #1}

\title{A Robust Optimization Approach for Scheduling with Uncertain Start-Time Dependent Costs}
\author[1]{Sof\'ia Rodr\'iguez-Ballesteros}
\author[1,2]{Javier Alcaraz\footnote{Corresponding author: \href{mailto:jalcaraz@umh.es}{jalcaraz@umh.es}, ORCID: \href{https://orcid.org/0000-0002-3526-508X}{https://orcid.org/0000-0002-3526-508X}.\\ 
\textit{E-mail addresses:} \href{mailto:sofia.rodriguezb@umh.es}{sofia.rodriguezb@umh.es}, \href{mailto:l.anton@umh.es}{l.anton@umh.es}, \href{mailto:marc.goerigk@uni-passau.de}{marc.goerigk@uni-passau.de}, \href{mailto:dorothee.henke@uni-passau.de}{dorothee.henke@uni-passau.de}.}}
\author[1,2]{Laura Anton-Sanchez}
\author[3]{Marc Goerigk}
\author[3]{Dorothee Henke}

\affil[1]{Centro de Investigación Operativa, Universidad Miguel Hernández de Elche, Spain.}
\affil[2]{Departamento de Estadística, Matemáticas e Informática, Universidad Miguel Hernández de Elche, Spain.}
\affil[3]{Business Decisions and Data Science, University of Passau, Passau, Germany.}

\date{}

\begin{document}
\linespread{1.25}

\maketitle

\begin{abstract}
In this work, we study a single-machine scheduling problem that aims at minimizing the total cost of a schedule subject to start-time dependent costs. This framework naturally captures scenarios where costs fluctuate throughout the day, such as time-varying energy or labor prices. To model more realistic scenarios, we assume that these costs lie within a budgeted uncertainty set and propose a two-stage robust optimization approach. In a first stage, the order in which activities should be executed is decided. After a cost scenario has been revealed, the starting times for each activity are established, subject to the ordering from the first stage. We demonstrate that the proposed problem is NP-hard and not approximable, implying the complexity of its robust counterpart. Furthermore, we show that already evaluating a first-stage solution is NP-hard when the uncertainty set is discrete. We develop models and solution methods for both continuous and discrete budgeted uncertainty. In computational experiments, we compare these approaches and demonstrate the advantages of including uncertainty beforehand.
\end{abstract}


% \newpage

\section{Introduction}

Scheduling problems play a fundamental role in manufacturing, service operations, process industries, and project management, where effective resource allocation and activity sequencing are essential for maintaining operational performance. Typical resources include equipment, raw materials, and labor, while activities range from product transformations and transportation to cleaning and maintenance operations. Scheduling objectives can take many forms, such as minimizing completion times, balancing workloads, or reducing operational costs.

Traditional scheduling approaches usually assume a completely deterministic environment, where key parameters---such as activity durations, resource availability, and constraints---are known in advance. Nonetheless, real-world scheduling is inherently uncertain, leading to frequent deviations from the planned schedule. Factors such as unpredictable weather, volatile markets, and fluctuating resource availability pose significant challenges. For instance, construction projects may experience delays due to adverse weather conditions, while manufacturing processes can be affected by material shortages. Similarly, service industries often face unexpected fluctuations in demand. Given these complexities, static deterministic scheduling approaches have been increasingly criticized (\citealp{Goldratt1997}), as they fail to account for uncertainty effectively. As a result, recent research has focused on integrating uncertainty into model parameters to better understand and mitigate its impact.

Numerous approaches have been proposed to address scheduling under uncertainty, and several dedicated typologies can be found in the literature (see, e.g., \citealp{suresh1993dynamic}; \citealp{Sabuncuoglu1999}; \citealp{chaari2014}). Notably, \cite{Herroelen2005} classify scheduling under uncertainty into five main approaches: reactive scheduling, stochastic scheduling, fuzzy scheduling, robust (proactive) scheduling, and sensitivity analysis. These methods fall under the broader categories of reactive or preventive scheduling. Reactive scheduling focuses on real-time adjustments, modifying an existing schedule during execution in response to unexpected disruptions such as order cancellations or machine breakdowns. In contrast, preventive scheduling seeks to anticipate uncertainty by incorporating historical data and forecasting techniques to model expected variability in processing times, demand, or costs. While adjustments may still be required over time, preventive scheduling provides a structured planning framework. Among preventive methods, we distinguish stochastic-based approaches, robust optimization methods, fuzzy programming, and sensitivity analysis.

In this work, we consider a  single-machine scheduling problem in which costs vary over a fixed time horizon and are affected by uncertainty. We first introduce our nominal problem, which is deterministic and, to the best of the authors’ knowledge, has not been addressed in the existing literature. The objective of this problem is to minimize the total cost of a schedule, subject to varying start-time dependent costs, i.e., the cost of initiating an activity varies over the time horizon. This framework captures real-world scenarios where operating costs---such as energy prices, labor wages, or disruption penalties---fluctuate over time. Examples include energy-aware manufacturing under time-of-use tariffs, infrastructure maintenance during varying traffic cycles, or satellite communications within orbital windows. In such environments, identifying an activity sequence that exploits lower-cost time windows constitutes a key operational challenge. We then extend this nominal problem to a two-stage robust scheduling problem by taking cost uncertainty into account. In the first stage, the decision maker only needs to determine the order in which activities shall be planned. Once the cost scenario has been revealed, the decision maker then determines the specific start time for each activity, while respecting the order that was fixed beforehand. Following \cite{Bertsimas2004}, we consider budgeted uncertainty sets, where we distinguish between two types: a discrete variant, which considers a finite set of cost scenarios, and a continuous variant, which allows convex combinations of such realizations.

This two-stage formulation arises whenever a sequence of activities must be determined in advance, while their specific start times allow for some flexibility. As an example, in airport scheduling (\citealp{ikli2021aircraft}), aircraft physically line up at the departure runway, which determines their departure order. The specific departure times, however, allow for some flexibility to adjust to the current circumstances. Similarly, consultation times---such as scheduled appointments with medical professionals---can often be delayed, but customers observe their order of arrival and consider changes in this sequence as unfair. Two-stage formulations similar to the one we adopt have also been explored in the project scheduling literature under uncertainty (see, e.g., \citealp{Artigues2013}; \citealp{BRUNI2017}; \citealp{BOLD2021}), where the sequence of activities or resource allocation decisions are fixed in the first stage, followed by scheduling adjustments in the second stage. For a recent overview of two-stage optimization approaches in this context, we refer to \cite{HAZIR2020}.

In robust scheduling, the general goal is to construct a baseline schedule that remains resilient against disruptions by ensuring good performance across all scenarios within a predefined uncertainty set. This approach aims to provide adaptability to varying conditions without relying on prior probabilistic assumptions. We refer to \cite{Herroelen2004} for a comprehensive understanding of the various procedures involved in generating robust schedules, as well as insights into the differences compared to reactive scheduling. Robust optimization methods have been widely applied to scheduling problems; examples include scheduling with budgeted uncertainty \citep{bougeret2019robust}, risk-averse decision criteria \citep{kasperski2019risk}, recoverable robustness \citep{bold2022investigating,bold2024recoverable}, and anchor-robustness \citep{bendotti2023anchor}. Chapter~9 of the book by \cite{robook} presents a recent overview. For an in-depth look at robust optimization in general, we refer to the books by \cite{ben2009} or \cite{bookbertsimasdenHertog2022}. In particular for single-machine scheduling problems, robust optimization has been used for a wide range of combinations between objective functions and uncertainty sets, e.g., \cite{yang2002robust,mastrolilli2013single,tadayon2015algorithms}.

Recall that, in our two-stage setting, once the scenario costs have been revealed, we fix starting times subject to an order in which activities need to be processed. As we show, this is closely related to a well-known project scheduling problem (PSP) with start-time dependent costs. This classical problem has received considerable attention in the literature (see, e.g., \citealp{Groflin}; \citealp{ChangEdmons}; \citealp{Roundy}; \citealp{Sankaran}). In particular, \cite{minimumCut} demonstrate that the PSP with start-time dependent costs can be solved in strongly polynomial time by transforming it into a minimum cut problem in a directed graph. This fact allows us to reformulate our two-stage robust model into a single-stage formulation.

Unlike traditional single-machine scheduling, which typically aims to minimize lateness costs or maximize earliness profits, our problem focuses on minimizing the total start-time dependent costs. These costs vary based on when an activity begins and are commonly found in project scheduling problems. Our single-machine scheduling problem  where the sequence of activities must be determined (see also, e.g., \citealp{baker1974}; \citealp{Pinedo2016}; \citealp{Shabtay2023}) might thus be seen as a special case of the resource-constrained PSP with a unit capacity-resource (see, e.g.. \citealp{Cauwelaert2020}; \citealp{Riedler2020}; \citealp{LIU2023}). 

This paper makes the following contributions. First, we introduce a novel single-machine scheduling formulation that considers start-time dependent costs and prove that this problem is NP-hard. We show that start-time dependent costs can effectively represent cumulative execution costs, which we use to establish our complexity results. Second, we present a new compact reformulation of the two-stage robust framework to handle continuous budgeted uncertainty. For discrete budgeted uncertainty, we demonstrate that the problem is more complex by proving NP-hardness of the adversarial problem. However, for this setting, we develop an iterative solution procedure inspired by \cite{zeng2013solving}. Third, we develop strengthening strategies for the exact robust approaches. For the compact formulation, we add valid inequalities, enforce transitivity dynamically, and propose warm-start schemes, which significantly improve root bounds, convergence, and final optimality gaps, thereby extending the range of instances solvable within the time limit. We transfer key ideas to the iterative discrete method as well, yielding faster convergence in moderately difficult cases. Finally, our computational study quantifies these gains and shows that the compact approach---especially in its strengthened variants---offers the best overall scalability and robustness--performance trade-off, remaining competitive and in common settings even superior under discrete adversarial evaluation.

The remainder of the paper is structured as follows: Section~\ref{sec2} introduces our novel problem, including a formal description of the robust setting employed in this study. This section presents both uncertainty sets and provides a detailed explanation of the two-stage robust approach. In Section~\ref{robust}, we present the proposed compact formulation for our problem under continuous budgeted uncertainty. Section~\ref{discrete} elaborates on the discrete scenario case, detailing the NP-hardness of the adversarial problem as well as the iterative method for solving the two-stage problem. Section~\ref{results} is dedicated to the presentation and analysis of the conducted experiments and their findings. Additional technical material and supplementary results are provided in Appendix~\ref{A}. Concluding remarks are presented in Section~\ref{conclusion}.

\section{Problem description}
\label{sec2}

This section formally introduces the scheduling problem addressed in this work. We begin by describing the nominal version of the problem and analyzing its computational complexity. Next, we incorporate uncertainty into the cost structure and propose a two-stage robust optimization approach to address the resulting problem.

\subsection{Nominal problem}
\label{nominal}

We study the following nominal problem, that is, our baseline problem which will be then extended by including uncertainty. Let a set $J$ of activities and a finite time horizon $\mathcal{T}=\{0,\ldots,T-1\}$ be given. Each activity $j\in J$ has an integer duration $d_j$. Furthermore, we are given a cost matrix $\textbf{c}$, where each entry $c_{jt}$ denotes the cost incurred by starting activity $j \in J$ at time $t \in \mathcal{T}$. Note that this formulation assumes that the entire cost matrix $\textbf{c}$ is part of the input, and hence $T$ is polynomial in the input size. This differs from most other scheduling problems, where the time horizon can be exponential in the input size.

Each activity needs to be scheduled to start at some point in time without preemption. No two activities can run in parallel. The goal is to minimize the sum of all starting costs. By introducing variables
$$ x_{jt}=\begin{cases}
1 & \text{if activity $j$ starts at time $t$,}\\
0 & \text{otherwise,}
\end{cases} $$
$\forall j\in J$ and $t\in\mathcal{T}$, a straight-forward integer programming formulation of the nominal problem is then as follows:
\begin{align}
& \text{minimize}
&& \sum_{j\in J} \sum_{t\in\mathcal{T}} c_{jt}x_{jt} && \label{(1)} \\
& \text{subject to}
&& \sum_{t \in \mathcal{T}} x_{jt} = 1, && \quad \forall j \in J, \label{(2)} \\
&&& d_jx_{jt} + \sum_{j\neq i\in J}\sum_{s=t}^{t+d_j-1} x_{is} \leq d_j, && \quad \forall j\in J, \ t\in\mathcal{T}, \label{(3)} \\
&&& x_{jt} \in \{0, 1\}, && \quad \forall j \in J, \ t \in \mathcal{T}. \label{(4)}
\end{align}
The objective function is the sum of all starting costs. By Constraints~\eqref{(2)}, it is ensured that each job starts once within the time horizon. Constraints (\ref{(3)}) ensure that once an activity is scheduled, it must be completed before another can begin, so that no two activities can overlap in execution. Note that this formulation resembles clique-based time-indexed models used in single-machine scheduling, where non-overlapping time intervals form cliques (see, e.g., \cite{cliqueFormulation} for exact algorithms in earliness-tardiness scheduling). However, unlike traditional approaches that focus on earliness-tardiness penalties, our model explicitly minimizes the sum of start-time dependent costs, adding a unique dimension to the problem. Furthermore, note that in this formulation, an activity may start late in the time horizon, such that the sum of its start time and duration exceeds $T-1$. For ease of notation, we allow this possibility. If undesired, additional constraints can be included to impose bounds on the earliest and latest starting times of each activity.

Note that the start-time dependent costs, $c_{jt}$, in our model can be motivated by scenarios where a specific cost, $w_{jt}$, must be paid whenever activity $j$ is active at time $t$, referred to as execution-time dependent costs. Such costs naturally arise, for example, in energy-intensive production or computing environments, where activities consume power while running and electricity prices vary over time. In such cases, the total cost of starting activity $j$ at time $t$ is simply the aggregate cost over its duration: $c_{jt} = \sum_{s=t}^{t + d_j -1} w_{js}$. This relationship allows us to use execution-time costs as a theoretical tool to prove the complexity of our problem while maintaining a general start-time dependent formulation for our robust models.

At this point, the question arises if this new problem is solvable in polynomial time. Note that the analysis of the problem complexity must take into account that the matrix $\textbf{c}$ is considered part of the input. The following result demonstrates the complexity of the nominal problem.

\begin{theorem}\label{prop:hardness} 
The nominal problem with execution-time dependent costs is NP-hard and not approximable in polynomial time, unless $P=NP$.
\end{theorem}

\begin{proof}
Let an instance of the strongly NP-hard 3-partition problem (see \cite{GJ79}) be given. It consists of a list $S$ of integers $a_1,\ldots,a_{3m}$ with $A=\sum_{i=1}^{3m} a_i/m$ and $A/4 < a_i < A/2$ for all $i\in\{1,\ldots,3m\}$. The task is to find a partition into $m$ triplets that all have the same sum.

We construct an instance of our scheduling problem as follows. We define activities $J = \{j_1, j_2,\ldots, j_{3m}\}$ with durations $d_j = a_j$. The time horizon is defined by $T=\sum_{j=1}^{3m} d_j + m + 1$. Finally, the execution-time dependent costs are defined as
\[ w_{jt} =\begin{cases}
    1 & \text{if } t = n(A+1) \text{ for some } n\in\mathbb{N}_{\ge 0}, \\
    0 & \text{otherwise.}
    \end{cases}
\]

Figure~\ref{fig:exampleproof} presents an example for $S=(3,3,3,3,4,4)$. We have $m=2$ and $A=10$. Each square represents a time slot. Black squares are those where $w_{jt}=1$. An optimal solution is to start activities $j_1$, $j_2$, and $j_5$ in the first 10 squares with costs zero, and start activities $j_3$, $j_4$, and $j_6$ in the remaining 10 squares with costs zero.

\begin{figure}
  \begin{center}
    \begin{tikzpicture}
      \newcommand{\drawsquares}[1]
      {
        \draw[step=0.5,thick] (0,0) grid (#1 / 2, 0.5);
      }
      \newcommand{\drawjob}[2]
      {
        \node at (-0.25,0.25) {$j_{#1}$};
        \drawsquares{#2}
      }

      \newcommand{\drawtimehorizon}
      {
        \drawsquares{22}
        \foreach \i in {0,...,22}
        {
          \node at (\i / 2 + 0.25, 0.75) {\small$\i$};
        }
        \foreach \i in {0,11,22}
        {
          \draw[fill=black] (\i / 2,0) rectangle (\i / 2 + 0.5,0.5);
        }
      }

      \begin{scope}
        \drawjob{1}{3}
      \end{scope}
      \begin{scope}[shift={(0,-1)}]
        \drawjob{2}{3}
      \end{scope}
      \begin{scope}[shift={(3,0)}]
        \drawjob{3}{3}        
      \end{scope}
      \begin{scope}[shift={(3,-1)}]
        \drawjob{4}{3}
      \end{scope}
      \begin{scope}[shift={(6,0)}]
        \drawjob{5}{4}
      \end{scope}
      \begin{scope}[shift={(6,-1)}]
        \drawjob{6}{4}
      \end{scope}
      \begin{scope}[shift={(-2,-2.5)}]
        \drawtimehorizon{}
      \end{scope}
    \end{tikzpicture}
\caption{Example illustrating the reduction in the proof of Theorem~\ref{prop:hardness}.}\label{fig:exampleproof}
\end{center}
\end{figure}
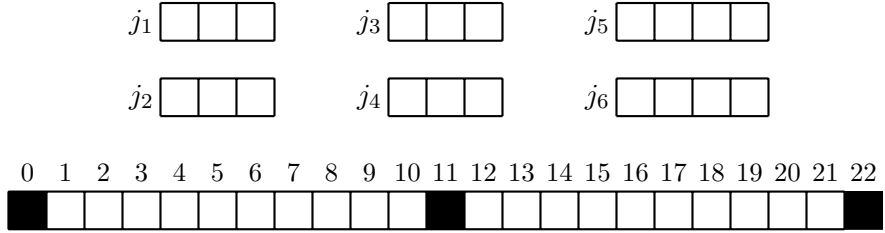

It follows that a schedule with costs 0 exists if and only if it is possible to partition the activities into $m$ sets $J_1,\ldots,J_m$ with $\sum_{j\in J_i} d_j = A$ for each $i \in \{1, \dots, m\}$. It remains to show that this reduction is of polynomial time and space in the input size. Note that there are $3m \cdot T$ cost coefficients that need to be defined, where $T$ is not polynomial in the binary encoding size of the input values $a_i$. However, as the 3-partition problem is strongly NP-hard, we may assume each $a_i$ to be encoded in unary, which means that the reduction remains polynomial.

Finally, the inapproximability claim follows from the fact that the 3-partition instance is a yes-instance if and only if a solution to the scheduling instance exists that has costs 0.
\end{proof}

Due to the relationship between start-time dependent costs and execution-time dependent costs explained above, Theorem~\ref{prop:hardness} directly implies:

\begin{corollary}
    The nominal problem with start-time dependent costs is NP-hard and not approximable in polynomial time, unless $P = NP$.
\end{corollary}

\subsection{Robust setting}
\label{uncertainty}

We now extend the nominal scheduling problem to a two-stage robust setting, where we assume that the cost matrix $\textbf{c}$ is uncertain. In the first stage, we do not determine the complete schedule (i.e., the starting time of each activity), but only the order in which activities are to be processed. Once the actual costs are revealed, the second stage assigns starting times to each activity, subject to the previously fixed ordering. Note that the nominal problem described in Section~\ref{nominal} can be recovered by assuming that both stages are solved simultaneously under perfect information. In this sense, the robust model extends the nominal formulation by separating sequencing and timing decisions, while incorporating uncertainty in start-time dependent costs.

More formally, let $y_{ij}$ be a binary variable that models if activity $i$ precedes activity $j$. To obtain a complete ordering, we use the constraints
\begin{align}
& y_{ii} = 0 & \forall i\in J\label{eq:ord0} \\
&y_{ij} + y_{ji} =1, & \forall i,j\in J, i\neq j\label{eq:ord1}\\
&y_{ij} + y_{jk} \le 1 + y_{ik}, & \forall i,j,k \in J, \text{pairwise different}. \label{eq:ord2}
\end{align}
Let $\mathcal{Y} = \{ \textbf{y} \in\{0,1\}^{J \times J} :  \eqref{eq:ord0}, \eqref{eq:ord1}, \eqref{eq:ord2} \}$. Furthermore, let $P(\textbf{y}) = \{ (i,j) \in J \times J : y_{ij} = 1\}$ for a given $\textbf{y}\in \mathcal{Y}$. The set of schedules that adhere to the precedence relations implied by $\textbf{y}$ is then defined as follows:
\begin{align*}
\mathcal{X}(\textbf{y}) = \Big\{ \textbf{x}\in\{0,1\}^{J\times \mathcal{T}} : \ & \sum_{t\in \mathcal{T}} x_{jt} = 1, & \forall j\in J,\phantom{\Big\}.} \\
& \sum_{s=t}^{T-1} x_{is} + \sum_{s=0}^{t+d_i-1} x_{js} \leq 1, & \forall (i,j)\in P(\textbf{y}),\ t\in\mathcal{T} \Big\}.
\end{align*}
Let $\mathcal{U}$ be an uncertainty set, which contains all scenarios $\textbf{c}$ against which we wish to immunize our solution. The two-stage robust optimization approach that we study is to solve
\[ \min_{\textbf{y}\in\mathcal{Y}} \ \max_{\textbf{c}\in\mathcal{U}} \ \min_{\textbf{x}\in\mathcal{X}(\textbf{y})} \ \sum_{j\in J} \sum_{t\in\mathcal{T}} c_{jt} x_{jt}. \]
Note that, if $\mathcal{U}$ consists of a single scenario, we recover the NP-hard nominal problem.

To define the uncertainty set, we follow the approach of budgeted uncertainty (\citealp{Bertsimas2004}). We assume that there is an interval $[\underline{c}_{jt},\underline{c}_{jt}+\hat{c}_{jt}]$ of possible realizations for the cost of starting activity $j$ at time $t$. In particular, we denote by $\delta_{jt}$, $j\in J, t\in\mathcal{T}$ the increase variables representing the scaled deviation from the lower bound $\underline{c}_{jt}$. We further assume that there is a budget $\Gamma$ on the total sum of these variables $\delta_{jt}$. That is, we define the continuous budgeted uncertainty set to be
    \begin{equation}\label{ContBU}
    \mathcal{U}_{C}(\Gamma) = \bigg\{\textbf{c}\in\mathbb{R}^{J \times \mathcal{T}}: c_{jt} = \underline{c}_{jt} + \hat{c}_{jt} \delta_{jt},\ \delta_{jt}\in [0,1],\  \sum_{j\in J}\sum_{t\in\mathcal{T}}\delta_{jt}\leq\Gamma \bigg\},
    \end{equation}
and the discrete budgeted uncertainty set as
    \begin{equation}\label{DiscBU}
    \mathcal{U}_{D}(\Gamma) = \bigg\{\textbf{c}\in\mathbb{R}^{J \times \mathcal{T}}: c_{jt} = \underline{c}_{jt} + \hat{c}_{jt} \delta_{jt},\ \delta_{jt}\in \{0,1\},\  \sum_{j\in J}\sum_{t\in\mathcal{T}}\delta_{jt}\leq\Gamma \bigg\}.
    \end{equation}  
If $\Gamma$ is an integer, it is well-known that both uncertainty sets are equivalent for one-stage robust problems. This is in general not true for two-stage problems: The adversary (i.e., the fictitious player solving the inner maximization problem) potentially has an advantage if costs of more than $\Gamma$ items are increased (\citealp{robook}).

We close this section with a discussion of extensions of this setting. The ``attack'' possibilities for an adversary in the above definition (i.e., the decision where to increase costs) are restricted to combinations of items and time slots. For example, it does not allow the adversary to perform one cost increase that affects all activities starting at a specific time, or all starting times of a specific activity. To this end, we can use scenarios of the form
\[ c_{jt} =  \underline{c}_{jt} + \hat{c}_{jt} \delta_{jt} + \tilde{c}_{jt} \tilde{\delta}_t + \bar{c}_{jt} \bar{\delta_j}, \]
with $\sum_{j\in J} \sum_{t\in\mathcal{T}} \delta_{jt} + \sum_{t\in\mathcal{T}} \tilde{\delta}_t + \sum_{j\in J} \bar{\delta_j} \le \Gamma$. Furthermore, it can also be extended to execution-time dependent uncertainty, by using
\[ c_{jt} = \sum_{s=t}^{t + d_j -1} \left( \underline{w}_{jt} + \hat{w}_{jt} \delta_{jt} + \tilde{w}_{jt} \tilde{\delta}_t + \bar{w}_{jt} \bar{\delta_j} \right). \]
To avoid additional notation, our presentation is restricted to the base case defined as $\mathcal{U}_{C}(\Gamma)$ and $\mathcal{U}_{D}(\Gamma)$, but can be extended to these more general settings as well.

\section{Proposed Methodology}\label{methodology}

This section outlines the exact approaches proposed for the two-stage robust single-machine scheduling problem introduced in Section~\ref{sec2}. We first derive a compact mixed-integer reformulation for continuous budgeted uncertainty by exploiting integrality and duality in the second-stage start-time assignment. We then address discrete budgeted uncertainty, for which we show that the adversarial problem is NP-hard and propose an iterative scenario-generation approach.

\subsection{Compact formulation for continuous budgeted uncertainty}
\label{robust}

We first recall the closely related PSP with start-time dependent costs as studied in \cite{mohring2001project,minimumCut}. We are given a set of precedence relations $P \subseteq J \times J$, where $(i,j) \in P$ means that activity $i$ must be scheduled before activity $j$. The task is to find a starting time for each activity, so that the total sum of start-time dependent costs is minimized, where no preemption of activities is allowed. The PSP with start-time dependent costs can be formulated as the following integer linear program:
\begin{align}
& \text{minimize} \quad
&& \sum_{j \in J} \sum_{t \in \mathcal{T}} c_{jt} x_{jt} \label{PSP-1} \\
& \text{subject to} \quad
&& \sum_{t \in \mathcal{T}} x_{jt} = 1, && \quad \forall j \in J, \label{PSP-2} \\
&&& \sum_{s = t}^{T-1} x_{is} + \sum_{s = 0}^{t + d_i - 1} x_{js} \leq 1, && \quad \forall (i, j) \in P, \ t \in \mathcal{T}, \label{PSP-3} \\
&&& x_{jt} \in \{0, 1\}, && \quad \forall j \in J, \ t \in \mathcal{T}. \label{PSP-4}
\end{align}
The objective of the problem is to minimize the total cost. Constraints~\eqref{PSP-3} model the precedence relationship, ensuring no activity starts before its predecessors have been completed. The complexity of this problem is quite well-understood. In fact, it can be solved as a linear program, as the following lemma states.

\begin{lemma}[\citealp{chaudhuri1994analyzing}]\label{lemma:integral}
	The polyhedron of problem~(\ref{PSP-1})-(\ref{PSP-4}) is integral.
\end{lemma}

This property states that it is possible to relax the integrality constraints~\eqref{PSP-4} to obtain a linear program (LP) that has an optimal integral solution. In particular, \cite{minimumCut} show that the problem can be solved in strongly polynomial time via a minimum cut formulation.

We now return to the inner problem
\[  \min_{\textbf{x}\in\mathcal{X}(\textbf{y})} \ \sum_{j\in J} \sum_{t\in\mathcal{T}} c_{jt} x_{jt}, \]
for given $\textbf{c}$ and $\textbf{y}$. It is thus the same as the problem (\ref{PSP-1})-(\ref{PSP-4}), which means that it can be written as an LP. Taking the dual of this model, we obtain the following program:
\begin{align}
& \text{maximize}
&& \sum_{j\in J} \alpha_j - \sum_{(i,j)\in P(\textbf{y})}\sum_{t\in\mathcal{T}} \gamma_{(i,j),t} \label{(8)}\\
& \text{subject to} 
&& \alpha_j - \sum_{\underset{r\in J}{(j,r)\in P(\textbf{y})}}\sum_{s=0}^{t} \gamma_{(j,r),s} - \sum_{\underset{r\in J}{(r,j)\in P(\textbf{y})}}\sum_{s= \text{max}\{0, t-d_r+1\}}^{T-1} \gamma_{(r,j),s} & \nonumber\\
&&& \leq c_{jt}, &  \forall j\in J, \ t\in\mathcal{T}, \label{(9)}\\
&&& \gamma_{(i,j),t}\geq 0, & \forall (i,j)\in P(\textbf{y}), \ t\in\mathcal{T}, \label{(10)} \\
&&& \alpha_j \ \text{free}, & \forall j\in J. \label{(11)}
\end{align}
To build the adversarial problem, we need to incorporate the optimization over the uncertainty set into the previous model. Precisely, we introduce the adversarial cost increase variables, $\delta_{jt}$, for all $j\in J$ and $t\in\mathcal{T}$, obtaining the following linear programming formulation:
\begin{align}
& \text{maximize}
&& \sum_{j\in J} \alpha_j - \sum_{(i,j)\in P(\textbf{y})}\sum_{t\in\mathcal{T}} \gamma_{(i,j),t} \label{(12)}\\
& \text{subject to} 
&& \alpha_j - \sum_{\underset{r\in J}{(j,r)\in P(\textbf{y})}}\sum_{s=0}^{t} \gamma_{(j,r),s} - \sum_{\underset{r\in J}{(r,j)\in P(\textbf{y})}}\sum_{s=\text{max}\{0, t-d_r+1\}}^{T-1} \gamma_{(r,j),s} & \nonumber\\
&&&  \leq \underline{c}_{jt} + \hat{c}_{jt} \delta_{jt}, & \forall j\in J, \ t\in\mathcal{T}, \label{(13)}\\
&&& \sum_{t\in\mathcal{T}}\sum_{j\in J} \delta_{jt}\leq\Gamma, & \label{(14)} \\ 
&&& 0\leq\delta_{jt}\leq 1, & \forall j\in J,\ t\in\mathcal{T}, \label{(15)}\\
&&&  \gamma_{(i,j),t}\geq 0, & \forall (i,j)\in P(\textbf{y}), \ t\in\mathcal{T}, \label{(17)} \\ 
&&& \alpha_j \ \text{free}, & \forall j\in J. \label{(18)}
\end{align}
Note that the dual variables $\alpha_j$, $\gamma_{(i,j),t}$, and the cost increase variables $\delta_{jt}$ are continuous. Hence, the adversarial problem is again an LP problem. By taking the dual of (\ref{(12)})-(\ref{(18)}), and combining it with the first-stage problem by means of the predecessor variables $\textbf{y}\in\mathcal{Y}$, we obtain the following mixed-integer programming compact formulation for the two-stage robust project scheduling with continuous budgeted uncertain start-time dependent costs:
\begin{align}
& \text{minimize}
&& \sum_{j\in J}\sum_{t\in\mathcal{T}} \underline{c}_{jt} x_{jt} + \Gamma \pi + \sum_{j\in J}\sum_{t\in\mathcal{T}}\eta_{jt} \label{(19)}\\
& \text{subject to} 
&& \sum_{t\in\mathcal{T}} x_{jt} = 1, & \forall j\in J, \label{(20)}\\
&&& \sum_{s=t}^{T-1} x_{is} + \sum_{s=0}^{t+d_i-1} x_{js} \leq 1 + (1-y_{ij}), & \forall i,j\in J,\ t\in\mathcal{T}, \label{(21)} \\ 
&&& \pi + \eta_{jt} \geq \hat{c}_{jt} x_{jt}, & \forall j\in J,\ t\in\mathcal{T}, \label{(22)}\\
&&& y_{ii} = 0, & \forall i\in J \\
&&& y_{ij} + y_{ji} = 1, & \forall i,j\in J, i\neq j\label{(24)}\\
&&& y_{ij} + y_{jk} \leq (1 + y_{ik}), & \forall i,j,k\in J, \text{pairwise different}\label{(25)} \\
&&& y_{ij}\in\{0,1\}, & \forall i,j\in J, \label{(26)}\\
&&& x_{jt}\geq 0, & \forall j\in J,\ t\in\mathcal{T}, \label{(27)} \\
&&& \eta_{jt} \geq 0, & \forall j\in J,\ t\in\mathcal{T}, \label{(28)} \\
&&& \pi \geq 0. \label{(29)} 
\end{align}
The predecessor binary variable $y_{ij}$ satisfies $y_{ij} = 1$ if activity $i$ is a predecessor of activity $j$; otherwise, $y_{ij} = 0$. Consequently, constraint (\ref{(21)}) is only activated when activity $i$ precedes activity $j$. Hence, the compact problem can be seen as an extended total cost minimization problem, in which a feasible schedule must be constructed with the objective of minimizing the overall cost in all the possible first and second-stage scenarios. Thus, the formulation implicitly optimizes over all activity permutations through the sequencing variables, jointly accounting for the corresponding worst-case cost realizations.

\subsection{Discrete budgeted uncertainty: complexity and iterative solution method}
\label{discrete}

We now consider the case of discrete budgeted uncertainty. To formulate the adversarial problem, we can follow the same approach as in the previous section. The only difference is that, in formulation (\ref{(12)})-(\ref{(18)}), the variables $\delta_{jt}$ are binary instead of continuous. 
\begin{align}
& \text{maximize}
&& \sum_{j\in J} \alpha_j - \sum_{(i,j)\in P(\textbf{y})}\sum_{t\in\mathcal{T}} \gamma_{(i,j),t} \label{(38)}\\
& \text{subject to} 
&& \alpha_j - \sum_{\underset{r\in J}{(j,r)\in P(\textbf{y})}}\sum_{s=0}^{t} \gamma_{(j,r),s} - \sum_{\underset{r\in J}{(r,j)\in P(\textbf{y})}}\sum_{s=\text{max}\{0, t-d_r+1\}}^{T-1} \gamma_{(r,j),s} & \nonumber\\
&&&  \leq \underline{c}_{jt} + \hat{c}_{jt} \delta_{jt}, & \forall j\in J, \ t\in\mathcal{T}, \label{(39)}\\
&&& \sum_{t\in\mathcal{T}}\sum_{j\in J}\delta_{jt}\leq\Gamma, & \label{(40)} \\ 
&&& \delta_{jt}\in\{0,1\}, & \forall j\in J,\ t\in\mathcal{T}, \label{(41)}\\
&&& \gamma_{(i,j),t}\geq 0, & \forall (i,j)\in P(\textbf{y}), \ t\in\mathcal{T}, \label{(42)}\\
&&& \alpha_j \ \text{free}, & \forall j\in J. \label{(43)}
\end{align}

This means that we cannot use LP-duality to reach a compact reformulation of the overall problem. In fact, the following proof shows that a specific variant of the adversarial problem is NP-hard. Theorem~\ref{theorem:adversarial} and Corollary~\ref{cor:adversarial} demonstrate that this is the case for both execution-time dependent costs and start-time dependent costs, as discussed in Section~\ref{nominal}. Moreover, we allow the adversary to increase the costs of a specific time slot, affecting any activity that is performed at this time, see also Section~\ref{uncertainty}.

\begin{theorem}\label{theorem:adversarial}
    For given $\textbf{y} \in \mathcal{Y}$, the adversarial problem with execution-time dependent costs and time-slot attacks, i.e., the problem
    \[\max_{\textbf{c} \in \tilde{\mathcal{U}}_D(\Gamma)}\ \min_{\textbf{x} \in \mathcal{X}(\textbf{y})}\ \sum_{j \in J} \sum_{t \in \mathcal{T}} w_{jt} x_{jt}\]
    with
    \[\tilde{\mathcal{U}}_D(\Gamma) = \bigg\{\textbf{w}\in\mathbb{R}^{J \times \mathcal{T}}: w_{jt} = \underline{w}_{jt} + \tilde{w}_{jt} \tilde{\delta}_{t},\ \tilde{\delta}_{t}\in \{0,1\},\ \sum_{t\in\mathcal{T}}\tilde{\delta}_{t}\leq\Gamma \bigg\},\]
      is NP-hard.
 \end{theorem}

 In order to prove Theorem~\ref{theorem:adversarial}, we first begin with a lemma. It is well-known that the min-max selection problem with discrete uncertainty is NP-hard, see \cite{kasperski2009randomized}. As a simple consequence, also the max-min problem version is hard.

\begin{lemma}\label{lemma:maxminsel}
The following max-min selection problem is strongly NP-complete: Given an uncertainty set $\mathcal{U} = \{\textbf{c}^1,\dots,\textbf{c}^K\}\subseteq\mathbb{N}^n_{0}$, and parameters $p,V\in\mathbb{N}$, is there a set $S \subseteq \{1, \dots, n\}$ with $\lvert S \rvert =p$ such that $\min_{k\in\{1, \dots, K\}} \sum_{i\in S} c^k_i \ge V$?
\end{lemma}

\begin{proof}

    Let an instance of the strongly NP-complete min-max selection problem with discrete uncertainty be given, consisting of an uncertainty set $\mathcal{U}'=\{\textbf{c}'^1,\dots,\textbf{c}'^{K'}\}\subseteq\mathbb{N}^{n'}_{0}$, and parameters $p',V'\in\mathbb{N}$. The question is if there is a set $S \subseteq \{1, \dots, n'\}$ with $\lvert S \rvert =p'$ such that $\max_{k\in\{1, \dots, K'\}} \sum_{i\in S} c'^k_i \le V'$.

    We construct an instance of the max-min selection problem by setting $n=n'$, $K=K'$, $p=p'$, and $c^k_i = C-c'^k_i$ for all $i \in \{1, \dots, n\}$ and $k \in \{1, \dots, K\}$, where $C = \max\{ \lceil V'/p \rceil + 1, \max \{ c'^k_i : i \in \{1, \dots, n'\}, k \in \{1, \dots, K'\}\}\}$. Finally, we set $V = pC - V'$.

    Then, for any $S\subseteq\{1, \dots, n\}$ with $\lvert S \rvert = p$, we have
\[ \min_{k\in\{1, \dots, K\}} \sum_{i\in S} c^k_i = \min_{k\in \{1, \dots, K\}} \sum_{i\in S} (C - c'^k_i) = pC - \max_{k\in\{1, \dots, K'\}} \sum_{i\in S} c'^k_i. \]
Hence, there is a set $S\subseteq\{1, \dots, n'\}$ with $\lvert S \rvert = p'$ and $\max_{k\in\{1, \dots, K'\}} \sum_{i\in S} c'^k_i \le V'$ if and only if there is a set $S\subseteq \{1, \dots, n\}$ with $\lvert S \rvert = p$ and $\min_{k\in\{1, \dots, K\}} \sum_{i\in S} c^k_i \ge pC - V'$.
\end{proof}

\begin{proof}[Proof of Theorem~\ref{theorem:adversarial}]

  Let an instance of the strongly NP-complete max-min selection problem (see Lemma~\ref{lemma:maxminsel}) with $n$ items, $K$ scenarios, and parameters $p$ and $V$ be given. We construct an instance of the adversarial problem as stated in Theorem~\ref{theorem:adversarial} in the following way. There are $|J|=nK+K+1$ jobs. Of these, there are $K+1$ so-called blocker jobs and $nK$ standard jobs. All jobs have a duration of $d_j = 1$. The given sequence of jobs (encoded in $\textbf{y} \in \mathcal{Y}$) is such that we begin with one blocker job, followed by $n$ standard jobs (which correspond to scenario $\textbf{c}^1$), followed by another blocker job, followed by $n$ standard jobs (which correspond to scenario $\textbf{c}^2$), and so on. In Figure~\ref{fig:adversarial}, we illustrate the given sequence in the top.

\begin{figure}
  \begin{center}
        \begin{tikzpicture}
      \newcommand{\drawsquares}[1]
      {
        \draw[step=0.5,thick] (0,0) grid (#1 / 2, 0.5);
      }

      \newcommand{\drawjobsequence}
      {
        \drawsquares{15}
        \foreach \i in {0,5,10,15}
        {
          \draw[fill=black] (\i / 2,0) rectangle (\i / 2 + 0.5,0.5);
        }
        \foreach \i in {1,2,3}
        {
          \node at (\i / 2 * 5 - 1,0.85) {$\textbf{c}^{\i}$};
        }
      }

      \newcommand{\drawtimehorizon}
      {
        \drawsquares{26}
        \foreach \i in {10,15}
        {
          \draw[fill=black] (\i / 2,0) rectangle (\i / 2 + 0.5,0.5);
        }
        \foreach \i in {11,...,14}
        {
          \draw[fill=lightgray] (\i / 2,0) rectangle (\i / 2 + 0.5,0.5);
        }
      }

      \begin{scope}
        \drawjobsequence
      \end{scope}
      \draw[-latex,very thick] (4,-0.5) -- (4,-1.5); 
      \begin{scope}[shift={(-2.5,-2.5)}]
        \drawtimehorizon
      \end{scope}
    \end{tikzpicture}
\end{center}
\caption{Illustration of the proof of Theorem~\ref{theorem:adversarial} for the case of $n = 4$ and~$K = 3$.}\label{fig:adversarial}
\end{figure}
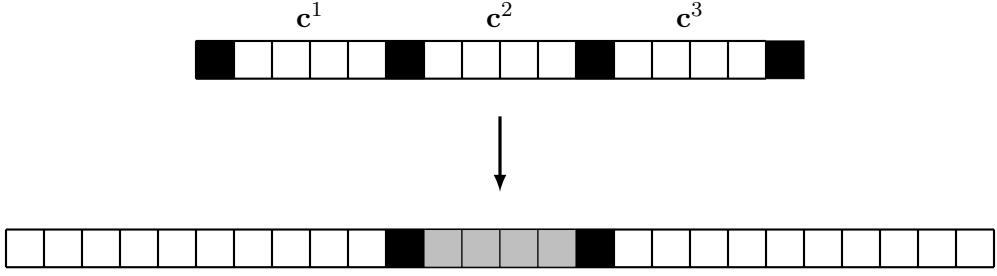

Blocker jobs have $\underline{w}_{jt} = \tilde{w}_{jt} = 0$ for all time slots $t$. In our time horizon (illustrated in the bottom of Figure~\ref{fig:adversarial}), there are two special time slots that are $n+1$ time slots apart. We call the $n$ time slots between the two special time slots the special region. Any standard job not assigned to the special region has high costs, i.e., we set $\underline{w}_{jt} = M$ (to be defined below) for standard jobs $j$ and time slots $t$ that are outside of the special region. Moreover, the uncertainty set does not allow any modification outside of the special region, i.e., $\tilde{w}_{jt} = 0$ for standard jobs $j$ and time slots $t$ that are outside of the special region. Note that this construction forces the inner minimization problem to choose a solution where two subsequent blocker jobs are assigned to the two special slots, as this maximizes the number of standard jobs in the special region and therefore minimizes the number of jobs for which the high costs $M$ are paid. In order to make any solution of this type possible, we choose the time horizon to be sufficiently large such that there are $n(K-1)+K$ time slots both left and right of the special region.

For all standard jobs, we set $\underline{w}_{jt} = 0$ within the special region.
For $i \in \{1, \dots, n\}$ and $k \in \{1, \dots, K\}$, let $j(i,k)$ be the $i$-th job in the part of the job sequence that corresponds to scenario $\textbf{c}^k$, and let $t(i)$ be the $i$-th time slot in the special region. We then define $\tilde{w}_{j(i,k), t(i)} = c^k_i$. Finally, we set $\Gamma=p$. This construction enables the adversarial decision maker to select $p$ time slots in the special region such that the jobs assigned to these time slots then have to pay the cost of the corresponding item $i$ in the scenario~$\textbf{c}^k$ that is chosen by the inner problem. Hence, the costs paid in the special region directly correspond to the costs in the max-min selection problem. Outside of the special region, an additional cost of $n(K-1)M$ occurs. Note that choosing $M = \sum_{k = 1}^K \sum_{i = 1}^n c^k_i + 1$ is sufficient for ensuring that the minimization problem selects a solution of the desired type.

We conclude that there exists a solution to the max-min selection problem with costs at least $V$ if and only if there exists a solution to the adversarial problem with costs at least $V + n(K-1)M$.
\end{proof}

Observe that the construction in the proof of Theorem~\ref{theorem:adversarial} defines the durations of all jobs to be~1. In this special case, there is no difference between start-time dependent costs and execution-time dependent costs. Hence, the following result holds:

\begin{corollary} \label{cor:adversarial}
The adversarial problem as defined in Theorem~\ref{theorem:adversarial}, but with start-time dependent costs instead of execution-time dependent costs, is NP-hard.
\end{corollary}

This result indicates that it is unlikely that a compact formulation for the robust problem with discrete budgeted uncertainty exists, such as the formulation that could be derived for the case of continuous budgeted uncertainty. Indeed, as the adversarial problem is NP-hard, it is unlikely that the robust problem is contained in NP at all. We conjecture the problem to be $\Sigma^p_2$-hard (see \citealp{goerigk2024complexity,grune2025completeness} for related complexity results), which indicates that it could be impossible to solve it directly using IP-solvers.

Due to this complexity, we propose an iterative solution method (\citealp{zeng2013solving}). For a subset $\mathcal{U}'=\{\textbf{c}^1,\ldots,\textbf{c}^K\}\subseteq \mathcal{U}_{D}(\Gamma)$ of scenarios, we determine a robust solution by solving
\[ \min_{\textbf{y}\in\mathcal{Y}}\ \max_{k\in\mathcal{K}}\ \min_{\textbf{x}\in\mathcal{X}(\textbf{y})}\ \sum_{j\in J} \sum_{t\in\mathcal{T}} c^k_{jt} x_{jt}, \]
where $\mathcal{K}=\{1,\ldots,K\}$. This can be reformulated as the following optimization problem:

\begin{align}
& \text{minimize}
&& z \label{(30)}\\
& \text{subject to} 
&& z \geq \sum_{j\in J}\sum_{t\in\mathcal{T}} c_{jt}^k x_{jt}^k, & \forall k\in \mathcal{K}, \label{(31)}\\
&&& \sum_{t\in\mathcal{T}} x_{jt}^k = 1, & \forall j\in J,\ k\in \mathcal{K}, \label{(32)}\\
&&& \sum_{s=t}^{T-1} x_{is}^k + \sum_{s=0}^{t+d_i-1} x_{js}^k \leq 1 + (1-y_{ij}), & \forall i,j\in J,\ t\in\mathcal{T},\ k\in \mathcal{K}, \label{(33)} \\ 
&&& y_{ii} = 0, & \forall i\in J, \\
&&& y_{ij} + y_{ji} = 1, & \forall i,j\in J, i\neq j\label{(34)}\\
&&& y_{ij} + y_{jl} \leq (1 + y_{il}), & \forall i,j,l\in J, \text{pairwise different}\label{(35)} \\
&&& x_{jt}^k\in\{0,1\}, & \forall j\in J,\ t\in\mathcal{T},\ k\in \mathcal{K}, \label{(36)}\\
&&& y_{ij}\in\{0,1\}, & \forall i,j\in J. \label{(37)} 
\end{align}
Solving (\ref{(30)})-(\ref{(37)}) provides us with a lower bound on the optimal value of the two-stage robust problem with respect to $\mathcal{U}_{D}(\Gamma)$, as only a subset of scenarios $\mathcal{U}'$ is considered. We obtain an upper bound by solving the adversarial problem (\ref{(38)})-(\ref{(43)}) with respect to this solution $\textbf{y}$. If the upper bound is larger than the previous lower bound, the corresponding scenario $\textbf{c}$ is added to $\mathcal{U}'$, and we repeat the process. Otherwise, lower and upper bounds coincide, which means that we have found an optimal solution. Note that $\mathcal{U}_{D}(\Gamma)$ is finite, which ensures finite convergence of this method.

\section{Empirical analysis}
\label{results}

In this section, we report computational experiments to assess the exact models and enhancements proposed in this paper. We begin by describing the benchmark instances and the experimental setup used in our study. We then analyze the performance of the baseline approaches---including the nominal formulations, the compact robust model, and the iterative discrete method---focusing on scalability and solution quality as problem size and uncertainty increase. Subsequently, we introduce and examine the effect of several strengthening strategies for the compact formulation. Finally, we transfer selected strengthening ideas to the iterative discrete approach and discuss the resulting improvements and remaining limitations.

\subsection{Experimental setup}

We begin by describing the benchmark instances generated for this study. Since the specific problem variant considered here excludes precedence constraints and lacks standard publicly available benchmarks (e.g., unlike \citealp{PSPLIB}), we created a new test suite. The primary instance characteristic is the number of activities, $n$. We generated 20 instances for each of the eight problem sizes $n \in \{5, 10, 15, 20, 25, 30, 35, 40\}$, resulting in a total of 160 instances. To isolate the impact of uncertainty, different values of the robustness parameter $\Gamma$ are considered for each problem size.

The instance parameters are defined as follows. Activity durations $d_j$ are drawn independently from a uniform distribution $U[1,5]$. The planning horizon is set to $T = 1.2 \cdot \sum_{j \in J} d_j$. Execution-time dependent costs $w_{jt}$ are generated from $U[1,9]$ and transformed into nominal start-time dependent costs $\underline{c}_{jt}$ as described in Section~\ref{nominal}. Cost deviations $\hat{c}_{jt}$ are drawn from $U[1,5]$.

We evaluate four distinct modeling approaches. First, we consider the nominal lower bound ($LB$), which solves the nominal formulation~\eqref{(1)}--\eqref{(4)} (see Section~\ref{nominal}) using nominal costs $\underline{c}_{jt}$. Second, we compute the nominal upper bound ($UB$) by solving the same nominal formulation with worst-case costs $\underline{c}_{jt} + \hat{c}_{jt}$. For the robust setting with continuous budgeted uncertainty, we evaluate the compact robust model ($C$), corresponding to the compact formulation~\eqref{(19)}--\eqref{(29)} introduced in Section~\ref{robust}. Finally, for discrete budgeted uncertainty, we employ the iterative discrete method ($ID$) described in Section~\ref{discrete}, which alternates between solving an adversarial subproblem and the associated master problem~\eqref{(30)}--\eqref{(37)}.

Model performance is assessed along two main dimensions: solution quality under uncertainty and computational efficiency. We report standard optimization metrics, including the best objective value found, the relative optimality gap, and the total runtime. In addition, we record the number of instances solved to proven optimality within the time limit, as well as the percentage of runs terminated due to time exhaustion. When the time limit is reached, the best feasible solution available is reported. Since a solution is defined by a fixed precedence ordering, its robustness is evaluated by computing the objective value under adversarial scenarios for both continuous ($Cont.\ BU$) and discrete ($Disc.\ BU$) uncertainty; lower values indicate higher robustness. Additionally, solutions are evaluated using the nominal lower-bound model ($LB$) in order to assess their performance in the absence of uncertainty.

The experimental analysis is structured to progressively investigate the sources of computational difficulty in the problem. We first analyze scalability as the number of activities $n$ and the uncertainty budget $\Gamma$ increase. In particular, this robustness parameter scales proportionally with the problem size. We then address the inherent computational hardness of the robust formulations by introducing and evaluating strengthened variants designed to improve convergence and scalability. These enhancements incorporate additional valid inequalities and warm-start strategies, and their impact is assessed in terms of both computational performance and solution quality.

All the models proposed in this paper have been coded in Java (\citealp{JAVA}) and solved using Gurobi 12.0.3. The computational experiments have been conducted on a cluster of workstations at Miguel Hernández University of Elche. In particular, we utilize the PRMTO-CIO-0 node, which is based on the Supermicro SYS-1029GQ-TRT model and equipped with two Intel(R) Xeon(R) Gold 6242R CPUs running at 3.10 GHz 
and 768 GB of RAM, running Linux. All instances were run in single-thread mode to avoid parallel execution and ensure fair time comparisons.

\subsection{Preliminary performance analysis}\label{preliminary}

This section provides a first comparative assessment of the different modeling approaches considered in this work, with the aim of identifying their computational behavior and establishing the main sources of difficulty of the problem. All computational results are obtained under a fixed time limit of two hours per instance. We start by analyzing two purely nominal formulations, namely the lower-bound ($LB$) and upper-bound ($UB$) models, which ignore uncertainty and serve as baseline references. We then examine the compact robust formulation ($C$), which explicitly incorporates continuous budgeted uncertainty through the parameter $\Gamma$, and assess its scalability and solution quality under increasing problem sizes and uncertainty levels. Finally, we consider the iterative discrete robust approach ($ID$) and compare its performance with that of the compact formulation. This progressive analysis allows us to disentangle the impact of uncertainty from that of the underlying scheduling structure and to highlight the trade-offs between robustness, computational effort, and solution quality.

Table~\ref{tab:nominal_LBUB} summarizes the computational performance of the nominal baseline models across all problem sizes considered. Both formulations solve every instance to optimality within negligible computational time, including the largest instances. This confirms that, in the absence of uncertainty, the underlying scheduling problem is computationally tractable and does not constitute a limiting factor, despite being NP-hard. Accordingly, all problem sizes in the computational study are solved using both the $LB$ and $UB$ models, and their solutions are used as reference points when assessing the behavior and performance of the robust formulations.

For each problem size, the table reports the average optimal objective value as reported by the model (\emph{Obj. value}) and the corresponding running time in seconds (\emph{$t$ (s)}). As expected, the upper-bound model consistently produces larger objective values than its lower-bound counterpart, since it optimizes against worst-case realizations of start-time dependent costs and therefore provides a conservative reference for the nominal execution cost.

\begin{table}[!ht]
    \small
    \centering
    \begin{tabular}{lcccccccc}
    \toprule
    & \multicolumn{8}{c}{\textbf{$LB$}} \\
    \cmidrule(lr){2-9}
    \textbf{$n$} & 5 & 10 & 15 & 20 & 25 & 30 & 35 & 40 \\
    \midrule
    \textbf{Obj. value} & 44.25 & 80.80 & 113.45 & 134.90 & 166.25 & 186.80 & 218.70 & 235.55 \\
    \textbf{$t$ (s)} & 0.00 & 0.01 & 0.07 & 0.15 & 0.30 & 0.61 & 4.02 & 6.53 \\
    \midrule
    & \multicolumn{8}{c}{\textbf{$UB$}} \\
    \cmidrule(lr){2-9}
    \textbf{$n$} & 5 & 10 & 15 & 20 & 25 & 30 & 35 & 40 \\
    \midrule
    \textbf{Obj. value} & 56.80 & 105.10 & 147.40 & 177.85 & 219.65 & 252.10 & 294.85 & 320.05 \\
    \textbf{$t$ (s)} & 0.00 & 0.01 & 0.07 & 0.14 & 0.32 & 0.64 & 4.78 & 7.34 \\
    \bottomrule
    \end{tabular}
    \caption{Average objective value and running time for the nominal lower- and upper-bound models ($LB$ and $UB$) over 20 instances per problem size.}
    \label{tab:nominal_LBUB}
\end{table}

It is worth stressing that these nominal formulations do not provide any control over uncertainty, as the robustness parameter $\Gamma$ is not present in either model. As a result, they merely capture two extreme cost scenarios---fully nominal and fully worst-case--- without regulating how many activities may be affected by adverse deviations. While useful as reference points, such solutions offer no intermediate level of protection. This limitation motivates the introduction of explicitly robust optimization models based on budgeted uncertainty. In the following, we therefore turn to the analysis of the compact robust formulation, which seeks to achieve a balanced compromise between solution quality and protection against uncertainty.

We now turn to the analysis of the compact robust formulation~\eqref{(19)}--\eqref{(29)}. While the nominal models are solved for all problem sizes and serve as baseline references throughout the computational study, the compact formulation is evaluated only for instance sizes for which robustness plays a meaningful role in shaping the solution. Based on preliminary computational experiments, instances with 5 and 10 activities were found to be solved very efficiently by the compact model, with limited sensitivity to the uncertainty budget. Consequently, the analysis of the compact formulation focuses on instances with 15 or more activities, where the impact of budgeted uncertainty becomes more pronounced and the behavior of the robust model can be more informatively assessed.

We start the analysis of the compact formulation by considering instances of small to moderate size, corresponding to problems with 15 and 20 activities. Table~\ref{tab:baseline_compact_1520} reports the average performance of the compact formulation under continuous budgeted uncertainty for these problem sizes and different values of the robustness parameter $\Gamma$.

For each problem size, the table is organized into blocks corresponding to different uncertainty levels ($u$), which represent the target proportion of activities that may be simultaneously affected by adverse cost deviations. Specifically, for both problem sizes we consider uncertainty levels of 30\%, 50\%, and 70\%. The uncertainty budget is computed as $\Gamma = \lceil u \cdot n \rceil$.
For each configuration, \emph{\#Opt.} denotes the number of instances (out of 20) solved to proven optimality within the time limit, while \emph{TL (\%)} reports the percentage of instances that reach the time limit.

The remaining columns summarize the main performance indicators of the method, reported as averages over the 20 instances. In particular, \emph{Obj. value} denotes the objective value of the best solution found, \emph{Root Bound} corresponds to the lower bound obtained at the root node of the branch-and-bound (B\&B) tree, \emph{Gap rel. (\%)} reports the relative optimality gap at termination, and \emph{$t$ (min)} gives the average computational time in minutes.

\begin{table}[!htb]
    \small
    \centering
    \adjustbox{max width=\textwidth}{
    \begin{tabular}{cccccccccc}
    \toprule
    & & & & & \multicolumn{4}{c}{\textbf{Avg.}}\\
    \cmidrule(lr){6-9}
    \textbf{$n$} & \textbf{$u$ (\%)} & \textbf{$\Gamma$} & \textbf{\#Opt.} & \textbf{TL (\%)} &\textbf{Obj. value} & \textbf{Root Bound} & \textbf{Gap rel. (\%)} & \textbf{$t$ (min)} \\
    \midrule
    \multirow{3}{*}{15} & 30 & \multicolumn{1}{c|}{5}  & 20 & 0 & 131.46 & 119.55 & 0.005 & 1.62 \\
        & 50 & \multicolumn{1}{c|}{8} & 20 & 0 & 136.86 & 122.69 & 0.005 & 3.41 \\
        & 70 & \multicolumn{1}{c|}{11} & 20 & 0 & 140.66 & 125.20 & 0.006 & 5.46 \\
    \midrule
    \multirow{3}{*}{20} & 30 & \multicolumn{1}{c|}{6}  & 20 & 0 & 155.90 & 143.01 & 0.006 & 24.16 \\
        & 50 & \multicolumn{1}{c|}{10} & 17 & 15 & 163.61 & 147.42 & 0.322 & 61.49 \\
        & 70 & \multicolumn{1}{c|}{14} & 11 & 45 & 168.59 & 150.84 & 1.227 & 81.97 \\
    \bottomrule
    \end{tabular}}
    \caption{Compact ($C$) formulation average results over 20 instances per configuration (small sizes).}
    \label{tab:baseline_compact_1520}
\end{table}

We first focus on instances with $n=15$ activities, for which the compact formulation is able to solve all instances to proven optimality within the time limit for all uncertainty levels considered. As shown in Table~\ref{tab:baseline_compact_1520}, average computational times remain relatively low, even when a large fraction of activities is allowed to deviate from their nominal costs. Nevertheless, a clear increase in computational effort is observed as the uncertainty budget grows. In particular, moving from an uncertainty level of 30\% to 50\% results in an average increase in solution time of more than 50\%, while increasing the uncertainty level from 30\% to 70\% leads to an average time increase of over 70\%. This trend highlights that, although instances with 15 activities remain computationally tractable, the complexity of the compact robust formulation is strongly affected by the level of protection against uncertainty.

The behavior of the compact formulation changes noticeably for instances with $n=20$ activities. While all 20 instances are solved to proven optimality within the time limit when the uncertainty level is set to 30\%, the computational burden increases substantially as the uncertainty budget grows. For an uncertainty level of 50\%, 3 out of the 20 instances exceed the two-hour time limit, whereas for the highest uncertainty level considered (70\%), 9 instances cannot be solved to optimality within the allotted time.

Despite this loss of optimality certification for larger uncertainty budgets, the quality of the solutions obtained by the compact formulation remains relatively high. In particular, the root bounds at the B\&B root node are fairly tight, and the relative optimality gaps at termination remain moderate. This suggests that the increased difficulty observed for $n=20$ is mainly due to the inability of the B\&B procedure to fully close the optimality gap within the time limit, rather than to a weak relaxation or poor incumbent solutions. Overall, these results indicate that, even for moderately sized instances, sufficiently large uncertainty budgets can significantly affect the computational performance of the compact robust formulation.

We next consider medium-sized instances, corresponding to problems with 25 and 30 activities. Table~\ref{tab:baseline_compact_2530} reports the performance of the compact formulation for these problem sizes. Since computational difficulty increases rapidly with both the number of activities and the robustness parameter $\Gamma$, the uncertainty levels explored in this experiment are adapted accordingly. Rather than fixing large uncertainty budgets, we progressively increase the uncertainty level starting from 10\% in order to identify the range over which the compact formulation remains computationally viable.

\begin{table}[!htb]
    \small
    \centering
    \adjustbox{max width=\textwidth}{
    \begin{tabular}{cccccccccc}
    \toprule
    & & & & & \multicolumn{4}{c}{\textbf{Avg.}}\\
    \cmidrule(lr){6-9}
    \textbf{$n$} & \textbf{$u$ (\%)} & \textbf{$\Gamma$} & \textbf{\#Opt.} & \textbf{TL (\%)} &\textbf{Obj. value} & \textbf{Root Bound} & \textbf{Gap rel. (\%)} & \textbf{$t$ (min)} \\
    \midrule
    \multirow{4}{*}{25} & 10  & \multicolumn{1}{c|}{3}  & 17 & 15 &  180.19 & 168.94 & 0.389 & 46.92 \\
        & 20 & \multicolumn{1}{c|}{5}  & 8  & 60 & 187.75 & 172.43 & 1.850 & 88.19 \\
        & 30 & \multicolumn{1}{c|}{8}  & 3 & 85 & 195.10 & 176.49 & 3.442 & 108.28 \\
        & 50 & \multicolumn{1}{c|}{13} & 0 & 100 & 205.25 & 181.67 & 5.980 & 120.00 \\
    \midrule
    \multirow{3}{*}{30} & 10 & \multicolumn{1}{c|}{3} & 16 & 20 &  201.12 & 190.00 & 0.601 & 71.80 \\
        & 20 & \multicolumn{1}{c|}{6}  & 1  & 95 & 213.85 & 195.36 & 4.396 & 118.39 \\
        & 30 & \multicolumn{1}{c|}{9}  & 0 & 100 & 223.83 & 199.41 & 7.029 & 120.00 \\
    \bottomrule
    \end{tabular}}
    \caption{Compact ($C$) formulation average results over 20 instances per configuration (breakdown sizes).}
    \label{tab:baseline_compact_2530}
\end{table}

For instances with $n=25$, a clear degradation in performance is observed as the uncertainty level increases. While 17 out of 20 instances are solved to proven optimality at the 10\% uncertainty level, this number drops to 8 at 20\% and to only 3 at 30\%. At the same time, the percentage of instances reaching the time limit increases sharply, from 15\% at 10\% uncertainty to 60\% and 85\% at 20\% and 30\%, respectively. When the uncertainty level reaches 50\%, none of the instances can be solved to optimality within the two-hour time limit, and all runs terminate due to time exhaustion.

A similar but even more pronounced behavior is observed for instances with $n=30$. Although 16 instances are solved to optimality at the 10\% uncertainty level, already at 20\% only a single instance reaches proven optimality, i.e., 95\% of the runs hit the time limit. For an uncertainty level of 30\%, no instance can be solved to optimality within the allotted time. In parallel, the average relative optimality gap increases substantially with the uncertainty level. For both problem sizes, gap values remain below 1\% at the lowest uncertainty level considered, but grow steadily as $\Gamma$ increases, reaching values close to 6\% for $n=25$ and exceeding 7\% for $n=30$ at the highest uncertainty levels.

Despite this deterioration, the root bounds obtained at the B\&B root node remain relatively tight across all configurations, indicating that the observed performance breakdown is primarily due to the difficulty of closing the optimality gap within the time limit rather than to a weak relaxation. An important insight emerging from these results is that the rapid deterioration in computational performance is driven primarily by the level of uncertainty rather than by the problem size alone. Indeed, already for $n=20$ activities, large uncertainty budgets lead to a substantial loss of optimality certification, and a similar pattern is observed for $n=25$ and $n=30$, where increasing $\Gamma$ ultimately prevents the compact formulation from solving any instance to proven optimality. Overall, these results clearly show that, for medium-sized instances, the compact formulation reaches its computational limits even for relatively modest uncertainty budgets, which strongly motivates the need for dedicated modeling enhancements. These enhancements are introduced and analyzed in detail in the following section.

We now turn to the evaluation of the quality of the solutions produced by the different approaches considered so far. To this end, we assess the performance of the schedules obtained by each model under a common evaluation framework that includes both nominal and adversarial cost realizations. In the following, we focus on instances with 15 and 20 activities, which represent the largest problem sizes for which the compact formulation is able to solve almost all instances to proven optimality within the time limit. These instance sizes therefore provide a natural and representative setting for comparing the quality of the solutions produced by the different approaches. Complete evaluation results for all tested configurations are reported in Appendix~\ref{A}.

\begin{figure}[!ht]
    \centering
    \includegraphics[width=0.9\textwidth]{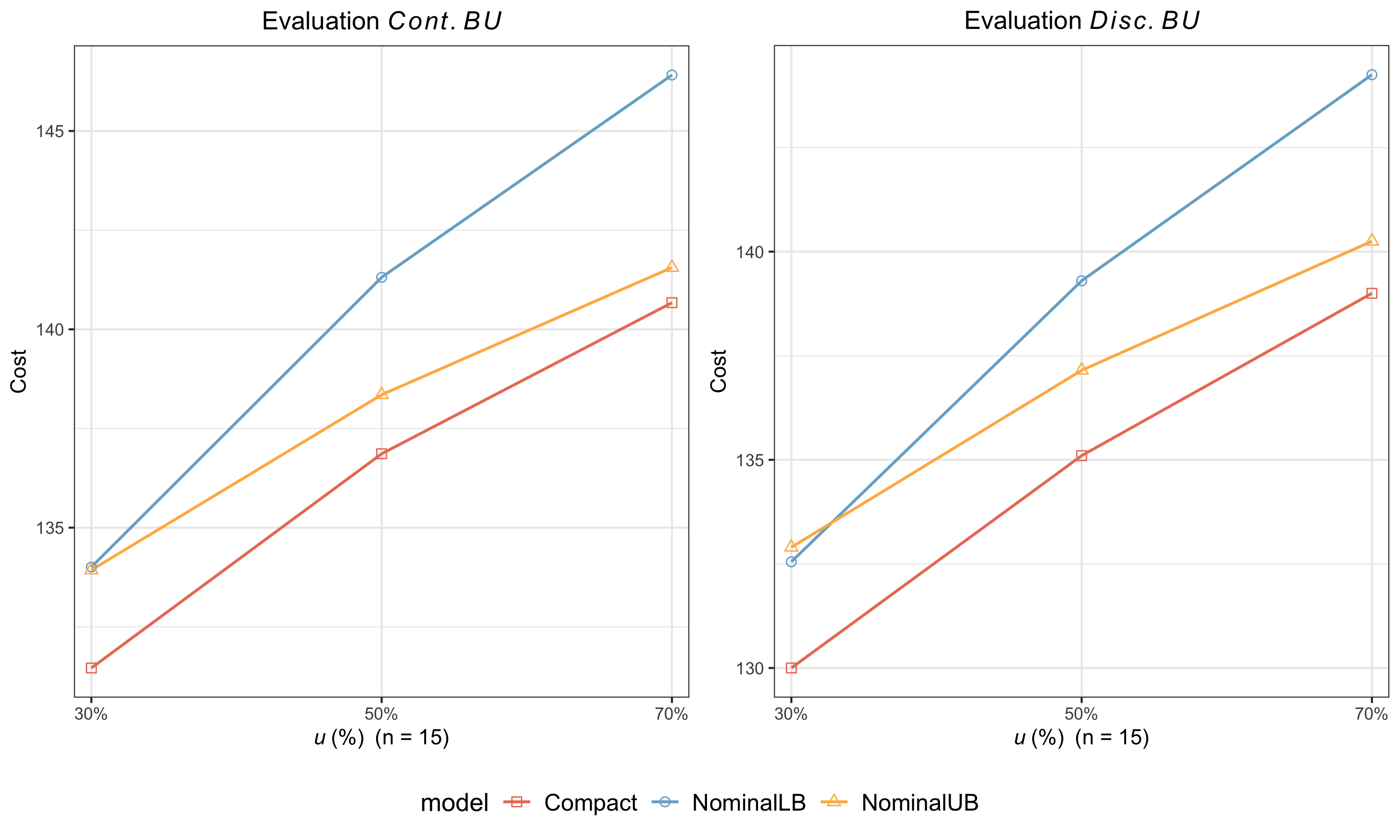}
    \includegraphics[width=0.9\textwidth]{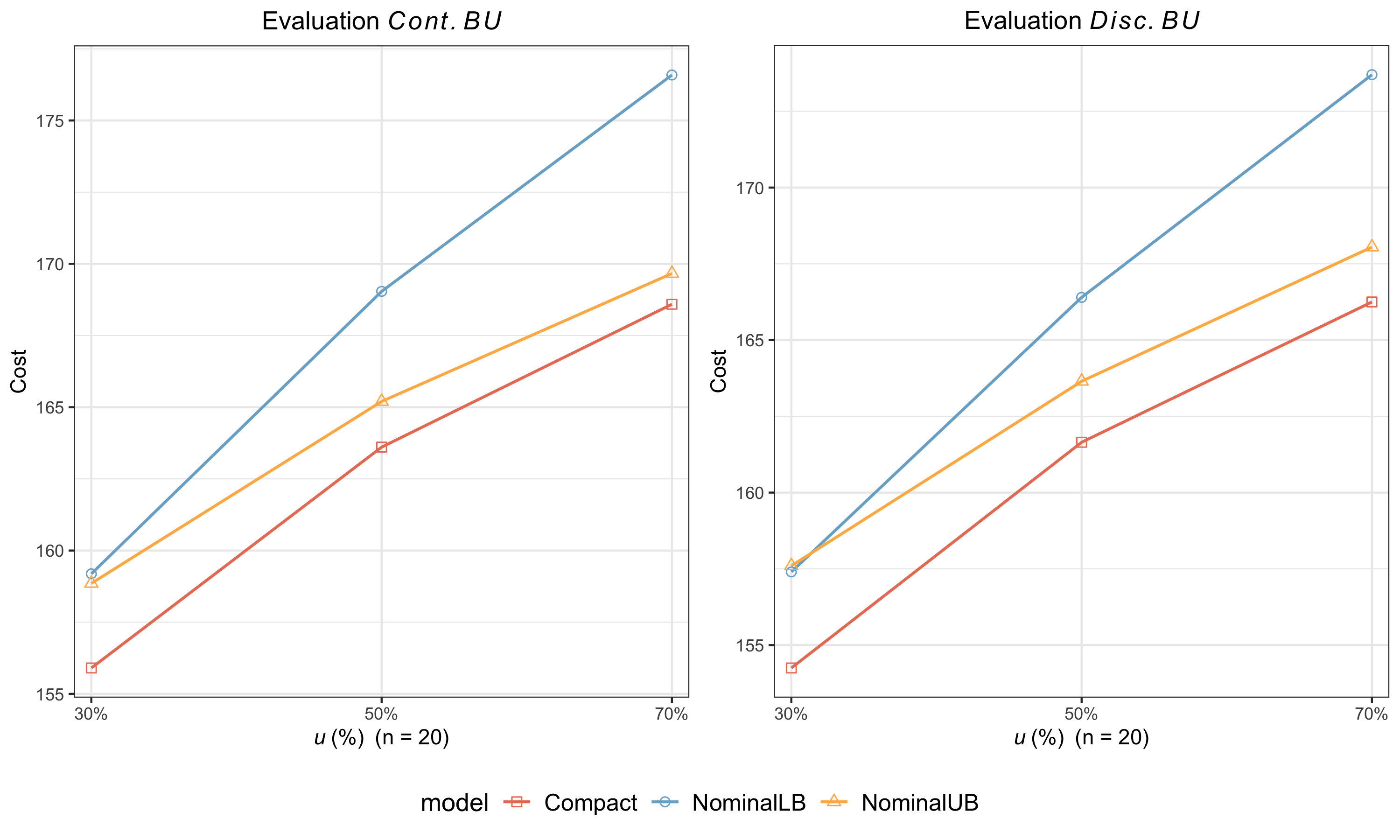}
    \caption{Average adversarial evaluation cost of the schedules obtained with $C$, $LB$ and $UB$ for increasing uncertainty levels and for instances with 15 and 20 activities.}
    \label{fig:preliminary_eval}
\end{figure}

Figure~\ref{fig:preliminary_eval} compares the average evaluation cost of the schedules produced by the different approaches under continuous and discrete adversarial budgeted uncertainty. Following Table~\ref{tab:baseline_compact_1520}, results are reported for uncertainty levels $u=30\%,50\%,70\%$.

In all cases, evaluation costs increase with the uncertainty level, reflecting the progressively more severe adversarial setting. In particular, the compact formulation consistently yields the lowest evaluation costs. By contrast, the nominal baselines exhibit a markedly different behavior. At low uncertainty levels ($u=30\%$), the $LB$ and $UB$ schedules attain similar evaluation costs, reflecting the limited impact of mild adversarial deviations. As uncertainty increases, however, the $LB$ formulation deteriorates rapidly, while the $UB$ approach remains more stable but still consistently underperforms relative to the compact model. This pattern reflects the fundamentally optimistic nature of $LB$ and the conservative design of $UB$.

Overall, these results indicate that explicitly incorporating budgeted uncertainty at the optimization stage leads to schedules that are more resilient under adversarial cost realizations. Unlike nominal formulations, which rely on either optimistic or fully pessimistic cost assumptions, the compact formulation directly limits how many activities may deviate simultaneously, resulting in solutions that achieve a better balance between robustness and conservatism.

\begin{table}[!ht]
\small
\centering
\begin{tabular}{cccccccc}
\toprule
 &  & 
& \multicolumn{1}{c}{\textbf{Eval. ($LB$)}} 
& \multicolumn{2}{c}{\textbf{Rel. change vs nominal (\%)}} 
& \multicolumn{2}{c}{\textbf{Eval. ($LB$)}} \\
\cmidrule(lr){4-4}\cmidrule(lr){5-6}\cmidrule(lr){7-8}
\textbf{$n$} & \textbf{$u$ (\%)} & \textbf{$\Gamma$}  & \multicolumn{1}{c}{$C$} 
& \multicolumn{1}{c}{$LB$} 
& \multicolumn{1}{c}{$UB$} 
& \multicolumn{1}{c}{$LB$} 
& \multicolumn{1}{c}{$UB$} \\
\midrule
\multirow{3}{*}{15} & 
30 & \multicolumn{1}{c|}{5}  & 115.45 & $-1.76$ & $3.79$ &  &  \\
& 50 & \multicolumn{1}{c|}{8}  & 117.05 & $-3.17$ & $2.46$ & 113.45 & 120.00 \\
& 70 & \multicolumn{1}{c|}{11} & 117.80 & $-3.83$ & $1.83$ &  &  \\
\midrule
\multirow{3}{*}{20} &
30 & \multicolumn{1}{c|}{6}  & 137.20 & $-1.70$ & $2.87$ &  &  \\
& 50 & \multicolumn{1}{c|}{10} & 139.55 & $-3.45$ & $1.20$ & 134.90 & 141.25 \\
& 70 & \multicolumn{1}{c|}{14} & 140.45 & $-4.11$ & $0.57$ &  &  \\
\bottomrule
\end{tabular}
\caption{Nominal $LB$ evaluation of compact solutions for different uncertainty levels and problem sizes. All values are averages over 20 instances.}
\label{tab:preliminary_nominalLB_eval}
\end{table}

Table~\ref{tab:preliminary_nominalLB_eval} reports the nominal performance of the solutions obtained with the compact formulation when evaluated under the nominal $LB$ model. The table compares the nominal evaluation of compact solutions with those obtained from the nominal $LB$ and $UB$ schedules, and reports the relative percentage change with respect to both benchmarks. Positive values indicate an improvement, whereas negative values correspond to a relative deterioration in nominal evaluation cost compared to the nominal model solutions.

The results confirm that explicitly incorporating uncertainty leads to a moderate deterioration in nominal performance when compared to the nominal $LB$ schedules. This deterioration increases with the uncertainty level, but remains below $5\%$ for all configurations considered. At the same time, the compact solutions consistently outperform the nominal $UB$. Taken together with the adversarial evaluation results, these findings indicate that the compact formulation achieves a meaningful trade-off: it accepts a controlled loss in nominal efficiency in exchange for a substantial gain in robustness against adverse cost realizations.

To conclude this section, we now turn to the case of discrete budgeted uncertainty and evaluate the iterative discrete approach ($ID$) described in Section~\ref{discrete}. In contrast to the continuous case, the adversarial problem involves binary cost increase variables, which prevents the derivation of a compact dual-based formulation. Consequently, $ID$ alternates between solving a master problem over a restricted scenario set and generating new worst-case scenarios through an adversarial subproblem.

This iterative structure results in a substantially higher computational burden than that of the compact formulation, as both the master and adversarial problems must be solved repeatedly over an expanding set of discrete scenarios. As a consequence, the scalability of the method is significantly more limited. Preliminary experiments show that, for problem sizes exceeding 15 activities, $ID$ fails to converge reliably within the imposed time limit, even for moderate uncertainty budgets. For this reason, the experimental analysis is restricted to instances with 5, 10, and 15 activities, which represent the largest problem sizes for which meaningful computational results can be consistently obtained.

Moreover, we tested a scenario enrichment strategy in which, in addition to the exact worst-case scenario for the current activity ordering, the adversary generates a small set of additional high-cost scenarios corresponding to the $K$ best candidates. These scenarios are added to the master problem in order to accelerate convergence by focusing the search on realizations that yield a significant bound improvement and are sufficiently distinct from those already considered, thereby avoiding the inclusion of redundant scenarios. Although this strategy marginally improves convergence in some cases, it does not alter the overall scalability behavior of the method nor the qualitative conclusions drawn from the baseline iterative approach.

\begin{table}[!htb]
    \small
    \centering
    \adjustbox{max width=\textwidth}{
    \begin{tabular}{ccccccccccc}
    \toprule
    & & & & & \multicolumn{4}{c}{\textbf{Avg.}}\\
    \cmidrule(lr){6-10}
    \textbf{$n$} & \textbf{$u$ (\%)} & \textbf{$\Gamma$} & \textbf{\#Opt.} & \textbf{TL (\%)} & \textbf{Obj. value} & \textbf{\#Iter.} & \textbf{Iter. best} & \textbf{Gap rel. (\%)} & \textbf{$t$ (min)} \\
    \midrule
    \multirow{2}{*}{5} 
        & 20 & \multicolumn{1}{c|}{1} & 20 & 0 & 47.75 & 2.90 & 1.85 & 0.000 & 0.01  \\
        & 30 & \multicolumn{1}{c|}{2} & 20 & 0 & 49.80 & 3.95 & 2.55 & 0.000 & 0.03  \\
    \midrule
    \multirow{2}{*}{10} 
        & 20 & \multicolumn{1}{c|}{2} & 19 & 5 & 88.25 & 7.70 & 4.25 & 0.064 & 14.98 \\
        & 30 & \multicolumn{1}{c|}{3} & 17 & 15 & 90.65 & 11.30 & 5.40 & 0.317 & 37.04 \\
    \midrule
    \multirow{2}{*}{15} 
        & 20 & \multicolumn{1}{c|}{3} & 4 & 80 & 125.20 & 9.15 & 4.75 & 2.643 & 112.55 \\
        & 30 & \multicolumn{1}{c|}{5} & 0 & 100 & 130.35 & 9.10 & 4.95 & 5.131 & 120.00 \\
    \bottomrule
    \end{tabular}
    }
    \caption{Iterative discrete solution method average results over 20 instances per configuration.}
    \label{tab:iterative_discrete}
\end{table}

Table 5 reports the average computational performance of the iterative discrete approach over 20 instances per configuration, for different problem sizes and uncertainty levels. In addition to standard indicators such as objective value, relative optimality gap, and computational time (in minutes), the table includes two iteration-related measures specific to the iterative nature of the method: the total number of iterations performed (\emph{\#Iter.}) and the iteration at which the best incumbent solution is found (\emph{Iter. best}). This distinction is relevant because the adversarial problem may yield higher objective values in later iterations, meaning that the best solution is not necessarily obtained at the final iteration.

Analyzing the results, we observe that for $n=5$ the iterative discrete approach consistently solves all instances to proven optimality within negligible computational time, regardless of the uncertainty level. In this setting, the number of iterations remains low and the relative optimality gaps are fully closed. As the problem size increases to $n=10$, the computational burden grows noticeably: although most instances are still solved to optimality (19 out of 20 for an uncertainty level of 20\% and 17 out of 20 for 30\%), the number of iterations increases substantially and the average computational time increases sharply.

The limitations of the iterative discrete approach become apparent for instances with $n=15$. Even under moderate uncertainty levels ($u=20\%$), only 4 instances are solved to proven optimality, and the average relative optimality gap remains sizable (2.6\%). For $u=30\%$, none of the instances can be solved to optimality within the allotted time. This behavior reflects the rapidly increasing cost of scenario generation and repeated master problem re-optimization as both the problem size and the uncertainty budget grow.

All schedules produced by $ID$ are evaluated under the nominal lower-bound model and under both adversarial problems ($\text{Cont.\ BU}$ and $\text{Disc.\ BU}$), exhibiting the same qualitative robustness--nominality trade-off observed for the compact formulation. The corresponding results are reported in Appendix~\ref{A} (Table~\ref{A:evaluation_id}).

The results discussed in this section indicate that, while both robust approaches perform well for small problem sizes, they become increasingly computationally demanding as the number of activities grows or as the level of uncertainty increases. In line with our theoretical complexity analysis, scalability limitations emerge more rapidly for $ID$, whereas the compact formulation exhibits more favorable computational behavior and remains tractable for larger instances, albeit with a growing burden at higher robustness levels. These observations motivate further improvements to both solution approaches through additional modeling and computational enhancements, which are discussed in the following sections.

\subsection{Enhancing the compact formulation: modeling and computational improvements}

Building on the computational insights obtained in the previous section, we now focus on strengthening the compact formulation to improve its performance for larger problem sizes and higher uncertainty budgets. To this end, we introduce a set of modeling and computational enhancements designed to reduce solution times and mitigate the growth of optimality gaps observed in challenging instances. In the following, we present a strengthened variant of the compact model and assess its computational impact.

The strengthened compact formulation ($C_{str}$) can be stated as follows. We first consider the incorporation of additional valid constraints into the baseline compact model~\eqref{(19)}--\eqref{(29)}. These constraints do not alter the set of integer-feasible solutions, but are designed to tighten the linear relaxation and accelerate the B\&B algorithm. In particular, all additional constraints are valid for the single-machine setting and are introduced solely to tighten the linear relaxation without excluding any integer-feasible solution.

The first strengthening constraint explicitly enforces the single-machine capacity at the time level:
\begin{equation}\label{(C1)}
\sum_{j \in J}\sum_{s=\max\{0,t-d_j+1\}}^{t} x_{js} \le 1,
\qquad \forall t \in \mathcal{T}.
\end{equation}
This constraint ensures that at most one activity can be processed at any given time. Although non-overlapping schedules are already enforced indirectly in the baseline formulation through sequencing constraints, the inclusion of this aggregated capacity constraint eliminates fractional solutions in which multiple activities appear partially active at the same time, thereby yielding a significantly stronger linear relaxation, particularly at the root node.

The second strengthening constraint targets the robust component of the formulation:
\begin{equation}\label{(C4)}
\eta_{jt} \le \hat{c}_{jt} x_{jt},
\qquad \forall j \in J,\ t \in \mathcal{T}.
\end{equation}
This constraint tightly links the deviation variables to the corresponding start-time decisions. While it is redundant for integer solutions, it prevents positive deviation values in the linear relaxation when the associated start time is not selected, thereby removing artificial flexibility in the robust layer and improving the convergence behavior of the B\&B algorithm.

Finally, we consider the transitivity constraints on the sequencing variables, as defined in Constraint~\eqref{(25)} of the baseline compact formulation,
\begin{equation}\nonumber
y_{ij} + y_{jk} \le 1 + y_{ik},
\qquad \forall i,j,k \in J \text{ pairwise different},
\end{equation}
which ensure consistency of the implied activity ordering. Since explicitly including all such constraints would lead to a cubic number of inequalities, we enforce transitivity dynamically by adding violated constraints as cutting planes during the B\&B process. This strategy preserves most of the strengthening effect of the transitivity constraints while keeping the formulation size computationally manageable.

Taken together, $C_{str}$ consists of the baseline compact model~\eqref{(19)}--\eqref{(29)}, augmented with constraints~\eqref{(C1)} and~\eqref{(C4)} and with transitivity constraints enforced dynamically through cutting planes. Lastly, we also consider a warm-started variant of $C_{str}$ to further accelerate the B\&B algorithm. Specifically, an initial incumbent solution is obtained by solving the nominal lower-bound model ($LB$) and provided to the solver at the start of the optimization process. This strategy does not alter the formulation itself, but can reduce the search effort by improving the initial bound.

Table~\ref{tab:strengthened_compact} summarizes the computational performance of the strengthened compact formulation, considering both the version without warm start (\emph{No WS}) and the variant initialized with the nominal lower-bound solution as a warm start (\emph{WS ($LB$)}). The table is organized into two main blocks corresponding to these two solution strategies and reports average results over 20 instances per configuration, using the same set of performance indicators as in the previous tables. Results are reported only for instances with 20, 25, and 30 activities, since for 15 activities the baseline compact formulation already exhibited very good computational performance, making further strengthening less informative. Both variants of $C_{str}$ are evaluated using the same uncertainty levels as those considered for $C$ in the previous section, ensuring that the effect of the proposed enhancements can be assessed on a directly comparable experimental setting.

\begin{table}[!htb]
    \small
    \centering
    \adjustbox{max width=\textwidth}{
    {\renewcommand{\arraystretch}{0.90}
    \begin{tabular}{ccc ccccc}
    \toprule
    & & & \multicolumn{5}{c}{\textbf{No WS}}\\
    \cmidrule(lr){4-8}
    \textbf{$n$} & \textbf{$u$ (\%)} & \textbf{$\Gamma$}
    & \textbf{\#Opt.} & \textbf{Obj. value} & \textbf{Root Bound} & \textbf{Gap rel. (\%)} & \textbf{$t$ (min)} \\
    \midrule
    \multirow{3}{*}{20} & 30 & \multicolumn{1}{c|}{6} & 20 & 155.90 & 147.94 & 0.004 & 6.29 \\
    & 50 & \multicolumn{1}{c|}{10} & 20 & 163.59 & 152.36 & 0.006 & 26.69 \\
    & 70 & \multicolumn{1}{c|}{14} & 17 & 168.36 & 155.72 & 0.279 & 42.24 \\
    \midrule
    \multirow{4}{*}{25} & 10 & \multicolumn{1}{c|}{3} & 20 & 180.06 & 175.25 & 0.001 & 9.68 \\
    & 20 & \multicolumn{1}{c|}{5} & 16 & 186.89 & 178.77 & 0.259 & 51.82 \\
    & 30 & \multicolumn{1}{c|}{8} & 6 & 194.46 & 182.73 & 1.614 & 99.54 \\
    & 50 & \multicolumn{1}{c|}{13} & 1 & 204.35 & 187.72 & 4.192 & 117.42 \\
    \midrule
    \multirow{3}{*}{30} & 10 & \multicolumn{1}{c|}{3} & 20 & 200.86 & 196.41 & 0.002 & 16.06 \\
    & 20 & \multicolumn{1}{c|}{6} & 5 & 211.90 & 201.98 & 1.562 & 106.04 \\
    & 30 & \multicolumn{1}{c|}{9} & 0 & 220.87 & 206.26 & 3.877 & 120.00 \\
    \midrule
     & & & \multicolumn{5}{c}{\textbf{WS ($LB$)}}\\
    \cmidrule(lr){4-8}
    \textbf{$n$} & \textbf{$u$ (\%)} & \textbf{$\Gamma$}
    & \textbf{\#Opt.} & \textbf{Obj. value} & \textbf{Root Bound} & \textbf{Gap rel. (\%)} & \textbf{$t$ (min)} \\
    \midrule
    \multirow{3}{*}{20} & 30 & \multicolumn{1}{c|}{6} & 20 & 155.90 & 147.94 & 0.004 & 6.26 \\
    & 50 & \multicolumn{1}{c|}{10} & 20 & 163.59 & 152.36 & 0.007 & 24.10 \\
    & 70 & \multicolumn{1}{c|}{14} & 17 & 168.31 & 155.72 & 0.259 & 45.29 \\
    \midrule
    \multirow{4}{*}{25} & 10 & \multicolumn{1}{c|}{3} & 20 & 180.06 & 175.25 & 0.002 & 8.89 \\
    & 20 & \multicolumn{1}{c|}{5} & 16 & 186.92 & 178.77 & 0.228 & 48.27 \\
    & 30 & \multicolumn{1}{c|}{8} & 7  & 194.41 & 182.73 & 1.484 & 98.47 \\
    & 50 & \multicolumn{1}{c|}{13} & 1  & 204.18 & 187.72 & 4.190 & 117.75 \\
    \midrule
    \multirow{3}{*}{30} & 10 & \multicolumn{1}{c|}{3} & 20 & 200.86 & 196.41 & 0.003 & 14.86 \\
    & 20 & \multicolumn{1}{c|}{6} & 5 & 212.18 & 201.98 & 1.681 & 106.05 \\
    & 30 & \multicolumn{1}{c|}{9} & 0 & 221.34 & 206.26 & 4.050 & 120.00 \\
    \bottomrule
    \end{tabular}}}
    \caption{Strengthened compact formulation average results over 20 instances per configuration, comparing runs without warm start (No WS) and with the nominal LB warm start (WS ($LB$)).}
    \label{tab:strengthened_compact}
\end{table}

From Table~\ref{tab:strengthened_compact}, it becomes apparent that the computational difficulty of the strengthened compact formulation is driven primarily by the level of uncertainty rather than by the problem size alone. In particular, when the uncertainty level is low ($u=10\%$), all instances are solved to proven optimality across all problem sizes considered, including the largest instances with $n=30$. As the uncertainty budget increases, the problem becomes more challenging, with solution times growing and relative gaps increasing. For $n=20$, however, the strengthened formulation remains highly effective even at the highest uncertainty level ($u=70\%$), yielding very small relative gaps (below 0.3\%) and solving nearly all instances to optimality. By contrast, for larger problem sizes ($n=25$ and $n=30$), higher uncertainty levels lead to a noticeable deterioration in performance, reflected in fewer optimal solutions and substantially larger gaps. For example, for the 30-activity instances with $u=30\%$, no instance is solved to optimality within the time limit, and the average relative gap exceeds 3.8\%.

The table also highlights the effect of incorporating a warm start based on the nominal lower-bound solution. Although this initial solution is obtained from an optimistic model that ignores cost uncertainty, it often leads to noticeable reductions in computational time while preserving comparable solution quality. For example, for $n=20$ and $u=50\%$, the average solution time decreases from 26.7 to 24.1 minutes when using the warm start, and similar time reductions are observed for larger problem sizes under low uncertainty levels, such as $n=25$ and $n=30$ with $u=10\%$. As the uncertainty budget increases, the benefit of the warm start becomes less systematic: in some cases it yields slightly smaller relative gaps, while in others the computational time remains unchanged or even increases, as observed for $n=20$ and $u=70\%$. Overall, these results indicate that the $LB$ warm start can be effective beyond purely low-uncertainty regimes, although its impact naturally diminishes as the robust component increasingly dominates the solution process.

Table~\ref{tab:improvements_compact_ws} summarizes the relative improvements achieved by the strengthened compact formulation with respect to the baseline compact model, expressed as percentage changes in objective value (\emph{$\Delta$ Obj.}), optimality gap (\emph{$\Delta$ Gap}), solution time (\emph{$\Delta$ Time}), and root bound (\emph{$\Delta$ Root}). Results are shown separately for \emph{No WS} and \emph{WS ($LB$)}, with positive values indicating improvements over the baseline formulation.

\begin{table}[!htb]
    \small
    \centering
    \adjustbox{max width=\textwidth}{
    \setlength{\tabcolsep}{7.5pt} 
    \begin{tabular}{ccc ccc ccc c}
    \toprule
    & & & \multicolumn{3}{c}{\textbf{No WS}} & \multicolumn{3}{c}{\textbf{WS ($LB$)}} &\\
    \cmidrule(lr){4-6}\cmidrule(lr){7-9}
    \textbf{$n$} & \textbf{$u$ (\%)} & \textbf{$\Gamma$} 
    & $\Delta\ Obj.$ & $\Delta\ Gap$ & $\Delta\ Time$
    & $\Delta\ Obj.$ & $\Delta\ Gap$ & $\Delta\ Time$ 
    & $\Delta\ Root$  \\
    \midrule
    \multirow{3}{*}{20} 
    & 30 & \multicolumn{1}{c|}{6} & 0.00 & 0.00 & 73.97 & 0.00 & 0.00 & 74.09 & 3.33 \\
    & 50 & \multicolumn{1}{c|}{10} & 0.01 & 98.14 & 56.59 & 0.01 & 97.83 & 60.81 & 3.24  \\
    & 70 & \multicolumn{1}{c|}{14} & 0.14 & 77.26 & 48.47 & 0.17 & 78.89 & 44.75 & 3.13 \\
    \midrule
    \multirow{4}{*}{25} 
    & 10 & \multicolumn{1}{c|}{3} & 0.07 & 99.74 & 79.37 & 0.07 & 99.49 & 81.05 & 3.60\\
    & 20 & \multicolumn{1}{c|}{5} & 0.46 & 86.00 & 41.24 & 0.44 & 87.68 & 45.27 & 3.55\\
    & 30 & \multicolumn{1}{c|}{8} & 0.33 & 53.11 & 8.07 & 0.35 & 56.89 & 9.06 & 3.41 \\
    & 50 & \multicolumn{1}{c|}{13} & 0.44 & 29.90 & 2.15 & 0.52 & 29.93 & 1.88 & 3.22\\
    \midrule
    \multirow{3}{*}{30} 
    & 10 & \multicolumn{1}{c|}{3} & 0.13 & 99.67 & 77.63 & 0.13 & 99.50 & 79.30 & 3.26\\
    & 20 & \multicolumn{1}{c|}{6} & 0.91 & 64.47 & 10.43 & 0.78 & 61.76 & 10.42 & 3.28 \\
    & 30 & \multicolumn{1}{c|}{9} & 1.32 & 44.84 & 0.00 & 1.11 & 42.38 & 0.00 & 3.32\\
    \bottomrule
    \end{tabular}}
    \caption{Relative improvements (in \%) of the strengthened compact formulation compared to the baseline compact model, reported separately without warm start (No WS) and with the nominal LB warm start (WS (LB)). Positive values indicate improvements (higher root bounds, lower optimality gaps, and shorter solution times).}
    \label{tab:improvements_compact_ws}
\end{table}

From the table, we observe that the most pronounced gains are achieved in terms of computational time and optimality gap. For low levels of uncertainty, solution times are reduced by more than 70\%, with peak reductions close to 80\%. At the same time, gap reductions are often close to 100\%, corresponding to configurations in which the strengthened compact formulation---both with and without warm start---attains proven optimality while the baseline compact model does not. For moderate uncertainty budgets, $C_{str}$ continues to yield substantial gap reductions; for instance, for $n=25$ and a 20\% level of uncertainty the gap reduction reaches 86\%, while at a 30\% uncertainty level it remains as high as about 53\%.

As the problem size and the robustness level increase, time reductions naturally diminish, since both formulations tend to exhaust the allocated $2h$ time limit in the most challenging settings. Nevertheless, even under these highly robust conditions, $C_{str}$ consistently achieves meaningful reductions in the final optimality gap, typically on the order of 30--45\%, reflecting a markedly improved convergence behavior. Notably, for the 20-activity instances, even the highest uncertainty level yields very strong improvements, with relative gap reductions above 77\% and computational time reductions close to 50\%.

Improvements in objective value remain modest throughout, with maximum gains around 1\%. In contrast, the root bound is systematically improved by slightly more than 3\% across all configurations, independently of the use of the warm start. This tighter linear relaxation provides a natural explanation for the observed reductions in optimality gaps and solution times. Regarding warm starting, the effect of \emph{WS ($LB$)} is generally secondary and not uniform across metrics, confirming that the primary performance gains stem from the strengthened formulation itself.

Overall, these results confirm that the proposed strengthening strategies substantially enhance the practical scalability of the compact formulation, allowing a much wider range of instances to be solved exactly and significantly improving solution quality and convergence behavior in challenging settings. To further explore the limits of these exact approaches, we finally consider larger instances with 35 and 40 activities. For these instances, the time limit is increased to 3 hours in order to account for the sharp growth in combinatorial complexity induced by both the problem size and the uncertainty budget. The analysis focuses on both strengthened compact formulations, with and without warm-starting from the nominal lower-bound solution, and considers uncertainty levels $u \in \{10\%, 20\%\}$. As in the previous experiments, performance is evaluated in terms of solution quality, optimality gaps, the number of instances solved to proven optimality, and required time.

\begin{table}[!htb]
    \small
    \centering
    \adjustbox{max width=\textwidth}{
    \begin{tabular}{ccc cccc cccc}
    \toprule
    & & & \multicolumn{4}{c}{\textbf{No WS}} & \multicolumn{4}{c}{\textbf{WS ($LB$)}} \\
    \cmidrule(lr){4-7}\cmidrule(lr){8-11}
    \textbf{$n$} & \textbf{$u$ (\%)} & \textbf{$\Gamma$}
    & \textbf{\#Opt.} & \textbf{Obj.} & \textbf{Gap rel. (\%)} & \textbf{$t$ (h)}
    & \textbf{\#Opt.} & \textbf{Obj.} & \textbf{Gap rel. (\%)} & \textbf{$t$ (h)} \\
    \midrule
    \multirow{2}{*}{35} 
    & 10 & \multicolumn{1}{c|}{4} & 13 & 237.82 & 0.314 & 1.72 & 14 & 237.80 & 0.264 & 1.61 \\
    & 20 & \multicolumn{1}{c|}{7} & 0 & 250.03 & 3.152 & 3.00 & 0 & 249.66 & 2.950 & 3.00 \\
    \midrule
    \multirow{2}{*}{40} 
    & 10 & \multicolumn{1}{c|}{4} & 12 & 254.29 & 0.591 & 2.14 & 11 & 254.20 & 0.599 & 2.09 \\
    & 20 & \multicolumn{1}{c|}{8} & 0 & 272.15 & 4.959 & 3.00 & 0 & 270.26 & 4.123 & 3.00 \\
    \bottomrule
    \end{tabular}}
    \caption{Average results over 20 instances for larger problem sizes (35 and 40 activities) with a time limit of 3 hours, comparing the improved compact formulation without warm start (No WS) and with the nominal LB warm start (WS ($LB$)).}
    \label{tab:improved_compact_3540}
\end{table}

The results reported in Table~\ref{tab:improved_compact_3540} illustrate the scalability limits of the strengthened compact formulation for larger problem sizes. For instances with 35 and 40 activities and a low uncertainty level ($u=10\%$), the formulation remains effective: 13--14 out of 20 instances are solved to proven optimality for $n=35$, and 11--12 instances for $n=40$, with small average optimality gaps (below 0.3\% and 0.6\%, respectively) and solution times well below the 3-hour limit. These results confirm that the proposed $C_{str}$ strategies substantially extend the range of instances that can be tackled exactly, enabling the solution of problem sizes that are already well beyond the reach of the baseline compact formulation under comparable conditions. As the uncertainty level increases to $u=20\%$, exact solvability deteriorates markedly, and no instance is solved to optimality within the time limit for either problem size. Nevertheless, high-quality incumbent solutions are consistently obtained, with average relative gaps around 3\% for $n=35$ and between 4\% and 5\% for $n=40$.

Overall, these results indicate that the practical limitation of the strengthened compact formulation is driven primarily by the level of uncertainty rather than by the number of activities alone. While large-scale instances can still be handled effectively under low uncertainty, high uncertainty budgets rapidly push the problem beyond the limits of exact optimization. At the same time, the ability of the proposed formulation to deliver high-quality solutions with relatively small optimality gaps in these challenging regimes highlights its practical relevance and robustness, providing a solid foundation for the development of complementary solution strategies for highly uncertain large-scale instances.

For completeness, we also assess the schedules obtained with the strengthened compact formulations, both \emph{No WS} and \emph{WS ($LB$)}, under the nominal lower-bound model and under the adversarial settings introduced earlier. A detailed discussion of these evaluations is provided in Appendix~\ref{A} (Table~\ref{A:evaluation_compact}), where a behavior consistent with that observed for the baseline compact formulation is reported.

\subsection{Extending the strengthening ideas to the iterative discrete approach}\label{improved-iterative}

The strengthening ideas introduced for the compact formulation can be partially transferred to the iterative discrete approach. In particular, the single-machine capacity constraint~\eqref{(C1)} is incorporated into the master problem~\eqref{(30)}--\eqref{(37)}, and transitivity is enforced dynamically through cutting planes, following the same rationale adopted for $C_{str}$. These enhancements aim at tightening the master problem relaxation and improving convergence by reducing the number of required iterations.

Tables~\ref{tab:iterative_discrete_improved} and~\ref{tab:improvements_iterative} summarize the computational performance of the strengthened iterative discrete approach ($ID_{str}$) and the corresponding relative improvements with respect to the baseline method. For instances with $n=10$, $ID_{str}$ solves almost all instances to proven optimality and yields a clear reduction in computational time, amounting to about $18\%$ for $u=20\%$ and exceeding $25\%$ for $u=30\%$. In contrast, objective values and relative optimality gaps remain essentially unchanged, indicating that the proposed enhancements primarily improve convergence speed without affecting solution quality for this problem size.

\begin{table}[!htb]
    \small
    \centering
    \adjustbox{max width=\textwidth}{
    \begin{tabular}{ccccccccccc}
    \toprule
    & & & & & \multicolumn{4}{c}{\textbf{Avg.}}\\
    \cmidrule(lr){6-10}
    \textbf{$n$} & \textbf{$u$ (\%)} & \textbf{$\Gamma$} & \textbf{\#Opt.} & \textbf{TL (\%)} & \textbf{Obj. value} & \textbf{\#Iter.} & \textbf{Iter. best} & \textbf{Gap rel. (\%)} & \textbf{$t$ (min)} \\
    \midrule
    \multirow{2}{*}{10} 
        & 20 & \multicolumn{1}{c|}{2} & 19 & 5 & 88.25 & 8.25 & 4.15 & 0.064 & 12.35 \\
        & 30 & \multicolumn{1}{c|}{3} & 17 & 15 & 90.65 & 11.25 & 5.60 & 0.317 & 27.68 \\
    \midrule
    \multirow{2}{*}{15} 
        & 20 & \multicolumn{1}{c|}{3} & 7 & 65 & 125.05 & 10.05 & 5.75 & 1.740 & 89.68 \\
        & 30 & \multicolumn{1}{c|}{5} & 1 & 95 & 130.25 & 10.30 & 5.80 & 4.471 & 119.99 \\
    \bottomrule
    \end{tabular}
    }
    \caption{Strengthened iterative discrete solution method average results over 20 instances per configuration.}
    \label{tab:iterative_discrete_improved}
\end{table}

For $n=15$, where the baseline iterative method already exhibits pronounced scalability limitations, the impact of the strengthening is more noticeable but still limited. At an uncertainty level of $u=20\%$, the number of instances solved to optimality increases from 4 to 7 (see Table~\ref{tab:iterative_discrete}), while the average relative optimality gap is reduced by more than $34\%$ and the average solution time decreases by approximately $20\%$. This indicates that tightening the master problem helps mitigate the degradation observed for moderately difficult instances.

At the higher uncertainty level ($u=30\%$), however, the benefits become marginal. The strengthened method is able to solve only one instance to proven optimality, reflecting that the intrinsic complexity of the iterative discrete framework dominates in highly uncertain settings. Nonetheless, compared to the baseline method, which fails to prove optimality for any of the 20 instances, $ID_{str}$ still achieves a reduction of more than $12\%$ in the average relative optimality gap. Overall, these results confirm that the strengthening improves convergence in moderately difficult cases, but does not fundamentally alter the scalability limitations of the approach.

\begin{table}[!htb]
    \small
    \centering
    \adjustbox{max width=\textwidth}{
    \setlength{\tabcolsep}{8pt}
    \begin{tabular}{ccc ccc}
    \toprule
    & & & \multicolumn{3}{c}{\textbf{Improvement}}\\
    \cmidrule(lr){4-6}
    \textbf{$n$} & \textbf{$u$ (\%)} & \textbf{$\Gamma$} 
    & $\Delta\ Obj.$ & $\Delta\ Gap$ & $\Delta\ Time$ \\
    \midrule
    \multirow{2}{*}{10}
    & 20 & \multicolumn{1}{c|}{2} & 0.00 & 0.00 & 17.56 \\
    & 30 & \multicolumn{1}{c|}{3} & 0.00 & 0.00 & 25.27 \\
    \midrule
    \multirow{2}{*}{15}
    & 20 & \multicolumn{1}{c|}{3} & 0.12 & 34.17 & 20.32 \\
    & 30 & \multicolumn{1}{c|}{5} & 0.08 & 12.86 & 0.01 \\
    \bottomrule
    \end{tabular}}
    \caption{Relative improvements (in \%) of the strengthened iterative discrete formulation compared to the baseline iterative method, for different problem sizes and uncertainty levels. Positive values indicate improvements in solution quality, optimality gap, and computational time.}
    \label{tab:improvements_iterative}
\end{table}

As in the previous formulations, the schedules produced by $ID_{str}$ are evaluated under the three models: \emph{LB}, \emph{Cont.\ BU}, and \emph{Disc.\ BU}. The corresponding results are reported in Appendix~\ref{A} (Table~\ref{A:evaluation_id}). In particular, we highlight the 15-activity instance set with an uncertainty level of $30\%$, which represents the largest tested scenario for which a direct comparison between the baseline compact formulation and the iterative discrete approaches ($ID$ and $ID_{str}$) is possible.

A closer inspection of the results reported in Appendix~\ref{A} (see Tables~\ref{A:evaluation_compact} and~\ref{A:evaluation_id}) reveals that, although the iterative discrete approaches are explicitly designed to address discrete budgeted uncertainty, they do not yield a significant improvement in solution quality when evaluated under \emph{Disc.\ BU}. In particular, both $ID$ and $ID_{str}$ attain very similar discrete adversarial evaluations, with average costs of 130.35 and 130.25, respectively. By comparison, $C$ achieves a slightly lower average discrete adversarial cost of 130.00 for the same instances.

This observation shows that, even under discrete adversarial evaluations, the compact formulation, which is originally developed for continuous budgeted uncertainty, produces solutions that are at least competitive with those obtained by the iterative discrete approach. When combined with its superior scalability and more stable convergence behavior, this result further highlights the effectiveness of the compact formulation as an exact solution method, even in settings where discrete uncertainty realizations are of interest.

Overall, the proposed strengthening strategies lead to consistent improvements in convergence and solution quality for the iterative discrete approach, whose adversarial subproblem is NP-hard (see Section~\ref{discrete}). Although the inherent complexity of discrete budgeted uncertainty ultimately limits scalability, the enhanced variants are able to exploit tighter relaxations and reduce computational effort in the range of instances where exact optimization remains viable.

\section{Conclusions}
\label{conclusion}
In this paper, we study a single-machine scheduling problem in which activity costs may vary throughout their processing time and are subject to uncertainty. We prove that this problem is NP-hard and NP-hard to approximate. Since the cost matrix is explicitly part of the input, the time horizon parameter $T$ is polynomial in the input size, which requires the use of a strongly NP-hard problem to establish the complexity result.

We then extend this setting to a two-stage robust optimization framework, where the ordering of activities is determined in the first stage and the corresponding start times are decided in the second stage once cost realizations are revealed. Exploiting the polynomial-time solvability of the scheduling problem when precedence constraints are fixed, we derive exact formulations for the adversarial problems under both continuous and discrete budgeted uncertainty. In the continuous case, the adversarial problem can be formulated as a linear program, which enables the construction of a compact mixed-integer formulation for the overall two-stage problem. For the discrete case, we demonstrate that the adversarial problem is NP-hard, and we therefore propose an iterative solution approach that alternates between solving the two-stage problem over a restricted set of scenarios and generating worst-case scenarios through the adversarial problem. As an open problem, we conjecture the robust problem with discrete budgeted uncertainty to be even $\Sigma^p_2$-hard.

In our computational study, we compare the nominal benchmarks ($LB$ and $UB$), the compact robust formulation under continuous budgeted uncertainty ($C$), and the iterative discrete approach for discrete budgeted uncertainty ($ID$). The nominal models are solved essentially instantaneously across all tested sizes up to $n=40$, confirming that the computational bottleneck stems from the robust counterparts rather than from the underlying single-machine scheduling structure. For the robust formulations, performance is primarily driven by the uncertainty budget $\Gamma$: while $C$ remains tractable on small instances even at high uncertainty, medium-sized instances increasingly reach the time limit as the uncertainty level grows. In contrast, $ID$ is empirically limited to small instances, as it requires repeated master re-optimization together with the solution of an NP-hard adversarial subproblem.

From a solution-quality perspective, evaluating schedules under a common framework (nominal $LB$ and adversarial \emph{Cont.\ BU} and \emph{Disc.\ BU}) shows that explicitly incorporating budgeted uncertainty at the optimization stage yields schedules that are substantially more resilient than nominal baselines. In particular, the compact approach consistently improves adversarial evaluations compared to $LB$ and $UB$, while the loss in nominal performance remains controlled (within a few percent in the tested settings). Moreover, compact solutions remain competitive relative to $UB$ also under nominal evaluation, confirming a practically meaningful robustness--nominality trade-off.

Finally, we assess the impact of strengthening strategies for both robust approaches. For the compact formulation, the strengthened variants (\emph{No WS} and \emph{WS (LB)}) significantly extend practical solvability: adding valid inequalities and enforcing transitivity dynamically yields systematic improvements in the root relaxation and translates into large reductions in runtime and optimality gaps, enabling the exact solution of substantially larger instances under low-to-moderate uncertainty and producing markedly tighter gaps when the time limit is reached. Warm-starting from $LB$ provides additional but secondary benefits. For discrete budgeted uncertainty, the strengthened iterative method ($ID_{str}$) improves convergence in moderately difficult cases, but does not fundamentally change scalability. Notably, even under discrete adversarial evaluation, the compact formulation produces solutions that are competitive with those obtained by the iterative approach, while offering superior scalability overall.

Several directions for future research follow naturally from our findings. A first avenue is to extend the model to richer scheduling settings, including additional resources, release dates, or multiple execution modes. A second direction is the development of tailored heuristic or metaheuristic methods for large-scale instances and high uncertainty budgets, where exact optimization becomes challenging. Finally, it would be natural to incorporate post-optimization flexibility through recoverable robustness, allowing limited adjustments of the first-stage ordering after costs are revealed and thereby strengthening practical applicability in highly uncertain environments.

\section*{Acknowledgments}
The authors thank the grants PID2021-122344NB-I00 and PID2022-136383NB-I00 funded by MICIU/AEI/ 10.13039/501100011033 and by ERDF/EU. This work was also supported by the Conselleria de Educación, Cultura, Universidades y Empleo, Generalitat Valenciana, Spain, under grants PROMETEO/2021/063, CIPROM/2024/34 and CIGE/2024/57.

\section*{Data and code availability statement}
All instances utilized in this study, together with the code used to generate the instances and to run the computational experiments, are available from the GitHub repository at the following link: \url{https://github.com/soofiarodriguez/Robust-Scheduling-Uncertainty}.

\appendix
\clearpage

\section{Full Tables of Computational Results}
\label{A}
\renewcommand{\thefigure}{A\arabic{figure}}  
\renewcommand{\thetable}{A\arabic{table}}    
\setcounter{figure}{0}  
\setcounter{table}{0} 
\renewcommand{\theHtable}{A\arabic{table}} 

This appendix reports the complete set of evaluation results for all solution approaches considered in the paper. Specifically, given a fixed schedule, we evaluate its nominal cost by solving the nominal lower-bound problem (\emph{LB}), and its worst-case performance by solving the two adversarial budgeted-uncertainty evaluation models introduced in Section~\ref{methodology}: the continuous adversary (\emph{Cont.\ BU}) and the discrete adversary (\emph{Disc.\ BU}). These evaluations provide a consistent basis for comparing nominal performance and robustness across the different optimization approaches.

\begin{table}[!htb]
    \small
    \centering
    \adjustbox{max width=\textwidth}{
    {\renewcommand{\arraystretch}{0.90}
    \begin{tabular}{ccc ccc ccc}
    \toprule
    & & & \multicolumn{6}{c}{\textbf{Evaluation}} \\
    \cmidrule(lr){4-9}
    & & & \multicolumn{3}{c}{\textbf{$LB$}} & \multicolumn{3}{c}{\textbf{$UB$}}\\
    \cmidrule(lr){4-6} \cmidrule(lr){7-9}
    \textbf{$n$} & \textbf{$u$ (\%)} & \textbf{$\Gamma$} &
    $LB$ & $Cont.\ BU$ & $Disc.\ BU$ &
    $LB$ & $Cont.\ BU$ & $Disc.\ BU$ \\
    \midrule

    \multirow{2}{*}{5}
        & 20 & \multicolumn{1}{c|}{1}  
        & \multirow{2}{*}{\centering 44.25} & 48.31 & 47.95 & 
        \multirow{2}{*}{\centering 44.95} & 48.41 & 48.05 \\
        & 30 & \multicolumn{1}{c|}{2}  &  & 51.09 & 50.30 &   & 50.91 & 50.15 \\
    \midrule

    \multirow{2}{*}{10}
        & 20 & \multicolumn{1}{c|}{2}  & 
        \multirow{2}{*}{\centering 80.80} & 89.91 & 89.20 & 
        \multirow{2}{*}{\centering 85.05 } & 92.04 & 91.05 \\
        & 30 & \multicolumn{1}{c|}{3}  &  & 93.56 & 92.25 &   & 94.17 & 93.05 \\
    \midrule

    \multirow{4}{*}{15}
        & 20 & \multicolumn{1}{c|}{3}  & 
        \multirow{4}{*}{\centering 113.45} & 127.11 & 126.35 & 
        \multirow{4}{*}{\centering 120.00} & 129.76 & 129.10 \\
        & 30 & \multicolumn{1}{c|}{5}  &   & 134.01 & 132.55 &   & 133.93 & 132.90 \\
        & 50 & \multicolumn{1}{c|}{8}  &   & 141.31 & 139.30 &   & 138.35 & 137.15 \\
        & 70 & \multicolumn{1}{c|}{11} &   & 146.41 & 144.25 &   & 141.56 & 140.25 \\
    \midrule

    \multirow{3}{*}{20}
        & 30 & \multicolumn{1}{c|}{6}  & 
        \multirow{3}{*}{\centering 134.90} & 159.23 & 157.40 & 
        \multirow{3}{*}{\centering 141.25} & 158.80 & 157.60 \\
        & 50 & \multicolumn{1}{c|}{10} &    & 169.07 & 166.35 &   & 165.16 & 163.60 \\
        & 70 & \multicolumn{1}{c|}{14} &    & 176.54 & 173.60 &   & 169.63 & 168.05 \\
    \midrule

    \multirow{4}{*}{25}
        & 10 & \multicolumn{1}{c|}{3}  & 
        \multirow{4}{*}{\centering 166.25} & 180.93 & 180.25 & 
        \multirow{4}{*}{\centering 175.95} & 186.85 & 186.00 \\
        & 20 & \multicolumn{1}{c|}{5}  &   & 189.60 & 187.75 &   & 191.65 & 190.55 \\
        & 30 & \multicolumn{1}{c|}{8}  &   & 200.14 & 197.10 &   & 197.49 & 196.05 \\
        & 50 & \multicolumn{1}{c|}{13} &   & 213.35 & 209.80 &   & 205.09 & 202.95 \\
    \midrule
    
    \multirow{3}{*}{30}
        & 10 & \multicolumn{1}{c|}{3}  & 
        \multirow{3}{*}{\centering 186.80} & 201.37 & 200.80 & 
        \multirow{3}{*}{\centering 199.75} & 210.51 & 209.40 \\
        & 20 & \multicolumn{1}{c|}{6}  &   & 214.17 & 212.45 &    & 218.02 & 216.55 \\
        & 30 & \multicolumn{1}{c|}{9}  &   & 224.65 & 222.05 &    & 224.03 & 222.30 \\
    \midrule
    
        \multirow{2}{*}{35}
        & 10 & \multicolumn{1}{c|}{4}  & 
        \multirow{2}{*}{\centering 218.70} & 238.41 & 237.75 & 
        \multirow{2}{*}{\centering 234.80} & 249.62 & 248.75 \\
        & 20 & \multicolumn{1}{c|}{7}  &  & 251.36 & 249.90 &  & 257.09 & 255.65 \\
    \midrule
    
    \multirow{2}{*}{40}
        & 10 & \multicolumn{1}{c|}{4}  & 
        \multirow{2}{*}{\centering 235.55} & 255.03 & 254.50 & 
        \multirow{2}{*}{\centering 251.90} & 266.99 &  265.55\\
        & 20 & \multicolumn{1}{c|}{8}  &   & 272.35 & 270.60 &  & 276.69 & 274.55 \\
    \bottomrule
    \end{tabular}}}
    \caption{Evaluation results for the nominal lower and upper bound optimization models, $LB$ and $UB$, respectively. For each $n$, the $LB$ value is constant across uncertainty levels and is therefore reported once per block.}
    \label{A:evaluation_nominal}
\end{table}

Table~\ref{A:evaluation_nominal} reports the evaluation results of the nominal lower-bound (\emph{LB}) and upper-bound (\emph{UB}) schedules under nominal costs as well as under continuous and discrete adversarial budgeted uncertainty. As expected, \emph{LB} yields the lowest nominal costs but deteriorates rapidly as the uncertainty level increases, particularly for larger instances. In contrast, \emph{UB} exhibits more stable adversarial evaluations due to its fully pessimistic cost assumptions, at the expense of consistently higher nominal costs. These results illustrate the inherent trade-off between nominal efficiency and robustness when uncertainty is not explicitly incorporated at the optimization stage.

\begin{table}[!htb]
    \small
    \centering
    \adjustbox{max width=\textwidth}{
    \begin{tabular}{ccc ccc ccc ccc}
    \toprule
    & & & \multicolumn{9}{c}{\textbf{Evaluation}} \\
    \cmidrule(lr){4-12}
    & & & & & & \multicolumn{6}{c}{$C_{str}$} \\
    & & & \multicolumn{3}{c}{\textbf{$C$}}
          & \multicolumn{3}{c}{\textbf{No WS}}
          & \multicolumn{3}{c}{\textbf{WS ($LB$)}} \\
    \cmidrule(lr){4-6} \cmidrule(lr){7-9} \cmidrule(lr){10-12}
    \textbf{$n$} & \textbf{$u$ (\%)} & \textbf{$\Gamma$} &
    $LB$ & $Cont.\ BU$ & $Disc.\ BU$ &
    $LB$ & $Cont.\ BU$ & $Disc.\ BU$ &
    $LB$ & $Cont.\ BU$ & $Disc.\ BU$ \\
    \midrule
    \multirow{3}{*}{15}
        & 30 & \multicolumn{1}{c|}{5}  & 115.45 & 131.46 & 130.00 & -- & -- & -- & -- & -- & -- \\
        & 50 & \multicolumn{1}{c|}{8}  & 117.05 & 136.86 & 135.10 & -- & -- & -- & -- & -- & -- \\
        & 70 & \multicolumn{1}{c|}{11} & 117.80 & 140.67 & 139.00 & -- & -- & -- & -- & -- & -- \\
    \midrule
    \multirow{3}{*}{20}
        & 30 & \multicolumn{1}{c|}{6}  & 137.20 & 155.90 & 154.25 & 137.20 & 155.90 & 154.25 & 138.53 & 157.24 & 155.47 \\
        & 50 & \multicolumn{1}{c|}{10} & 139.55 & 163.61 & 161.65 & 139.25 & 163.59 & 161.55 & 139.40 & 163.59 & 161.55 \\
        & 70 & \multicolumn{1}{c|}{14} & 140.45 & 168.59 & 166.25 & 140.35 & 168.36 & 166.15 & 140.35 & 168.31 & 166.05 \\
    \midrule

    \multirow{4}{*}{25}
        & 10 & \multicolumn{1}{c|}{3}  & 167.20 & 180.19 & 179.15 & 167.00 & 180.06 & 179.15 & 167.00 & 180.06 & 179.15 \\
        & 20 & \multicolumn{1}{c|}{5}  & 167.95 & 187.01 & 185.37 & 168.85 & 186.89 & 184.95 & 169.10 & 186.92 & 185.10 \\
        & 30 & \multicolumn{1}{c|}{8}  & 171.20 & 195.10 & 192.30 & 170.80 & 194.45 & 191.80 & 170.50 & 194.41 & 191.65 \\
        & 50 & \multicolumn{1}{c|}{13} & 173.95 & 205.03 & 201.95 & 174.00 & 204.35 & 201.20 & 173.75 & 204.11 & 201.30 \\
    \midrule

    \multirow{3}{*}{30}
        & 10 & \multicolumn{1}{c|}{3}  & 187.25 & 201.12 & 200.40 & 187.10 & 200.86 & 200.05 & 187.05 & 200.86 & 200.10 \\
        & 20 & \multicolumn{1}{c|}{6}  & 190.80 & 213.85 & 212.15 & 188.80 & 211.88 & 210.20 & 188.60 & 212.17 & 210.60 \\
        & 30 & \multicolumn{1}{c|}{9}  & 195.75 & 224.04 & 221.00 & 192.90 & 220.82 & 218.25 & 191.45 & 221.26 & 219.10 \\
    \midrule

    \multirow{2}{*}{35}
        & 10 & \multicolumn{1}{c|}{4}  & -- & -- & -- & 219.15 & 237.82 & 237.20 & 219.20 & 237.80 & 237.05 \\
        & 20 & \multicolumn{1}{c|}{7}  & -- & -- & -- & 222.25 & 250.03 & 247.90 & 220.25 & 249.64 & 247.55 \\
    \midrule

    \multirow{2}{*}{40}
        & 10 & \multicolumn{1}{c|}{4}  & -- & -- & -- & 236.10 & 254.27 & 253.60 & 235.95 & 254.20 & 253.50 \\
        & 20 & \multicolumn{1}{c|}{8}  & -- & -- & -- & 242.05 & 272.00 & 269.35 & 237.55 & 270.24 & 267.80 \\
    \bottomrule
    \end{tabular}}
    \caption{Evaluation results for the compact formulation ($C$) and its strengthened ($C_{str}$) variants, with (WS ($LB$)) and without (No WS) warm start. }
    \label{A:evaluation_compact}
\end{table}

Table~\ref{A:evaluation_compact} summarizes the evaluation results for the schedules obtained with the compact formulation ($C$) and its strengthened variants ($C_{str}$), both with and without warm starting, namely \emph{WS (LB)} and \emph{No WS}, respectively. Across all tested configurations, the compact solutions exhibit a balanced performance profile, achieving substantially lower evaluation costs under \emph{Cont.\ BU} and \emph{Disc.\ BU} than the nominal $LB$ and $UB$ schedules, while incurring only a moderate loss in nominal performance. The strengthened variants preserve this trade-off and, in some cases, slightly improve adversarial evaluations, confirming that the proposed enhancements primarily affect convergence behavior and solution quality without materially altering the robustness characteristics of the resulting schedules.

\begin{table}[!htb]
    \small
    \centering
    \adjustbox{max width=\textwidth}{
    {\renewcommand{\arraystretch}{0.90}
    \begin{tabular}{ccc ccc ccc}
    \toprule
    & & & \multicolumn{6}{c}{\textbf{Evaluation}} \\
    \cmidrule(lr){4-9}
    & & & \multicolumn{3}{c}{\textbf{$ID$}} & \multicolumn{3}{c}{$ID_{str}$}\\
    \cmidrule(lr){4-6} \cmidrule(lr){7-9}
    \textbf{$n$} & \textbf{$u$ (\%)} & \textbf{$\Gamma$} &
    $LB$ & $Cont.\ BU$ & $Disc.\ BU$ &
    $LB$ & $Cont.\ BU$ & $Disc.\ BU$ \\
    \midrule

    \multirow{2}{*}{5}
        & 20 & \multicolumn{1}{c|}{1}  & 44.70 & 48.40 & 47.75 & -- & -- & -- \\
        & 30 & \multicolumn{1}{c|}{2}  & 44.95 & 50.88 & 49.80 & --  & -- & -- \\
    \midrule

    \multirow{2}{*}{10}
        & 20 & \multicolumn{1}{c|}{2}  & 81.85 & 89.80 & 88.25 & 81.80 & 89.71 & 88.25 \\
        & 30 & \multicolumn{1}{c|}{3}  & 82.45 & 92.53 & 90.65 & 82.40 & 92.56 & 90.65 \\
    \midrule

    \multirow{2}{*}{15}
        & 20 & \multicolumn{1}{c|}{3} & 114.60 & 126.61 & 125.20 & 115.20 & 126.86 & 125.05 \\
        & 30 & \multicolumn{1}{c|}{5} & 115.90 & 132.57 & 130.35 & 115.70 & 132.42 & 130.25 \\
    \bottomrule
    \end{tabular}}}
    \caption{Evaluation results for the iterative discrete method ($ID$) and its strengthened variant ($ID_{str}$).}
    \label{A:evaluation_id}
\end{table}

Table~\ref{A:evaluation_id} reports the evaluation results for the iterative discrete approach ($ID$) and its strengthened variant ($ID_{str}$). Overall, both methods produce very similar nominal and adversarial evaluations, with only marginal improvements observed for $ID_{str}$. These improvements reflect improved convergence of the strengthened method, which enables it to reach slightly better solutions or tighter optimality gaps. As discussed in Section~\ref{improved-iterative}, although these approaches are specifically designed to address discrete budgeted uncertainty, their adversarial evaluations (under \emph{Disc.\ BU}) remain comparable to those obtained with the compact formulation for the 15-activity instances with $\Gamma = 5$, which constitute the only configuration solved by all three methods ($C$, $ID$, and $ID_{str}$).

Overall, the evaluation results reported in this appendix complement the analysis presented in the main text by providing a detailed and comprehensive comparison of nominal and robust solution behaviors across all tested configurations. They confirm that explicitly incorporating budgeted uncertainty at the optimization stage leads to solutions that achieve a more favorable balance between nominal performance and robustness, while also highlighting the practical limitations of approaches that rely on purely nominal or fully pessimistic assumptions.

\end{document}